\documentclass[reqno]{amsart}

\usepackage{amscd,amsmath,amssymb,amsfonts}
\usepackage{mathrsfs}
\usepackage[dvips]{color} 
\usepackage[all]{xy}

\numberwithin{equation}{section}

\theoremstyle{plain}
    \newtheorem{thm}{Theorem}[section]
    \newtheorem{lem}[thm]{Lemma}
    
    \newtheorem{prop}[thm]{Proposition}

\theoremstyle{definition}    
    
    \newtheorem{exmp}[thm]{Example}

\def\Coker{\mathrm{Coker}}
\def\GL{\mathrm{GL}}

\def\dR{{\mathrm{d\hspace{-0.2pt}R}}}            

\def\Per{\mathrm{Per}}
\def\id{{\mathrm{id}}}              
\def\Image{{\mathrm{Im}}}        
\def\Hom{{\mathrm{Hom}}}  
\def\Ext{{\mathrm{Ext}}}

\def\MHdR{{\mathrm{MHdR}}}  
\def\BdR{{\mathrm{BdR}}}  
\def\FiltBdR{\text{\rm Filt-BdR}} 
  
\def\ker{{\mathrm{Ker}}}          
\def\Ker{{\mathrm{Ker}}}          
\def\top{\mathrm{top}}
\def\Pic{{\mathrm{Pic}}}

\def\End{{\mathrm{End}}}
\def\pr{{\mathrm{pr}}}

\def\pr{{\mathrm{pr}}}
\def\reg{{\mathrm{reg}}}          %
\def\Res{\mathrm{Res}}
\def\Gr{{\mathrm{Gr}}}
\def\rank{{\mathrm{rank}}}
\def\zar{{\mathrm{zar}}}
\def\bA{{\mathbb A}}

\def\C{{\mathbb C}}

\def\P{{\mathbb P}}

\def\Q{{\mathbb Q}}

\def\Z{{\mathbb Z}}

\def\cH{{\mathscr H}}
\def\cM{{\mathscr M}}
\def\cD{{\mathscr D}}

\def\cJ{{\mathscr J}}

\def\O{{\mathscr O}}

\def\cU{{\mathscr U}}
\def\vg{\varGamma}

\def\lra{\longrightarrow}

\def\hra{\hookrightarrow}

\def\ot{\otimes}
\def\op{\oplus}
\def\wt#1{\widetilde{#1}}

\def\ol#1{\overline{#1}}

\def\os#1#2{\overset{#1}{#2}}

\def\Aut{\mathrm{Aut}}
\def\cg{G_\mu}

\def\a{\alpha}\def\b{\beta}\def\m{\mu}
\def\pl{\partial_\lambda}
\def\l{\lambda}
\def\FF#1#2#3{{}_2F_1\left({#1 \atop #2};#3\right)}
\def\FFF#1#2#3{{}_3F_2\left({#1 \atop #2};#3\right)}
\def\GG#1#2{\Gamma\left({#1 \atop #2}\right)}

\def\lam{a}

\begin{document}
\title{Regulators of $K_1$ of Hypergeometric Fibrations}
\author{Masanori Asakura and Noriyuki Otsubo}
\address{Department of Mathematics, Hokkaido University, Sapporo, 060-0810 Japan}
\email{asakura@math.sci.hokudai.ac.jp}
\address{Department of Mathematics and Informatics, Chiba University, Chiba, 263-8522 Japan}
\email{otsubo@math.s.chiba-u.ac.jp}

\begin{abstract}
We study a deformation of what we call hypergeometric fibrations. 
Its periods and $K_1$-regulators are described in terms of hypergeometric functions 
${}_3F_2$ in a variable given by the deformation parameter. 
\end{abstract}

\date{\today}

\maketitle

\section{Introduction}\label{intro-sect}
In \cite{a-o-1} and \cite{a-o} we studied the periods and regulators
for a certain class of fibrations, which we call {\it hypergeometric fibrations} (see \S \ref{HG-sect}
for the definition).
The purpose of this paper is to extend the main results in \cite{a-o}.


Let $f:X\to \P^1$ be a hypergeometric fibration in the sense of \S \ref{HG-sect1}.
Let $\pi:\P^1\to\P^1$ be a map given by $t\mapsto t^l$ with $l\geq 1$ an integer.
Let 
\[
\xymatrix{
X^{(l)}\ar[r]^{i\quad}\ar[rd]_{f^{(l)}}&X\times_{\pi,\P^1}\P^1\ar[r]\ar[d]&X\ar[d]^f\\
&\P^1\ar[r]^\pi&\P^1
}
\]
be a Cartesian diagram with $i$ a desingularization.
One of the main results in \cite{a-o} is the period formula which describes
the periods of $X^{(l)}$, and the other is the regulator formula
which describes Beilinson's regualtor map
on the motivic cohomology group $H^3_\cM(X^{(l)},\Q(2))$, especially on elements
supported on certain singular fibers of $f^{(l)}$.
In particular we described the regulator in terms of the special values of
the generalized hypergeometric functions ${}_3F_2\left({\a_1,\a_2,\a_3,
\atop\beta_1,\beta_2};z\right)$ at $z=1$.

We extend those results in the following way. Our idea is simple, just replacing $\pi$ with a map $\pi_\lambda$ given by $t\mapsto \lambda-t^l$ for $\lambda\in \C\setminus\{0,1\}$. 
We then obtain fibrations $f^{(l)}_\lambda:X^{(l)}_\lambda\to \P^1$ in the same way as above, 
and they are parametrized by $\lambda$.
We discuss the periods and regulators for $X^{(l)}_\lambda$.
Since the fibtarions are parametrized by $\lambda$,
the periods and regulators are no longer complex numbers but analytic functions.
The main results of this paper
are to describe them in terms of hypergeometric functions 
(Theorems \ref{main1}, \ref{main2}).
         
Taking the limits $\lambda\to 0$ of mixed Hodge structures ($\Leftrightarrow$ 
the nearby cycle cohomology functor $\psi_{\lambda=0}$),
one can derive the main results of \cite{a-o} from our main results.
However we make somewhat a strong assumption
``$\alpha^\chi_1\in\Z$'' throughout this paper, so that
they do {\it not} cover all of \cite{a-o}.

Our another motivation is the {\it logarithmic formula} in \cite{a-o-t}
where we gave a sufficient condition for that the special value of ${}_3F_2$ at $z=1$ is
written by a linear combination of log of algebraic numbers.
Theorem \ref{reg-comp-cor1} (=a precise version of Theorem \ref{main2})
enables us to obtain its functional version, namely
we can give a sufficientl condition for that ${}_3F_2(z)$
is written in terms of the logarithmic functions. 
This will be discussed in a paper \cite{a-o-log}.


At the conference ``Regulator IV'' in Paris (May 2016), 
S. Bloch asked the first author whether results in the author's talk gave examples
to the following question of V. Golyshev.
\begin{description}
\item[Question]
Let \[P_{\mathrm{HG}}=D_z\prod_{i=1}^{p-1}(D_z+\beta_i-1)-z\prod_{i=1}^p(D_z+\alpha_i),
\quad D_z:=z\frac{d}{dz}\]
be the hypergeometric differential operator and let $M=D_S/D_SP_{\mathrm{HG}}$
the $D_S$-module on $S=\P^1\setminus\{0,1,\infty\}$ where
$D_S$ denotes the sheaf of differential operators.
Suppose that $M$ is reducible, equivalently $\exists\alpha_i\in\Z$
or $\alpha_j-\beta_k\in\Z$ for some $j,k$, so that there is an exact sequence
\[
0\lra N\lra M\lra Q\lra0
\]
of $D_S$-modules.
Then does it underly a variation of mixed Hodge structures of geometric origin?
If so, does the extension data arise from Beilinson's regulator map on a
motivic cohomology group? Moreover, is the regulator described in terms of hypergeometric functions
which are solutions of $P_{\mathrm{HG}}$?
\end{description}
See Theorem \ref{reg-comp-cor2}.  Our regulator formula (Theorem \ref{main2})
gives an affirmative answer in case $p=3$ and $\alpha_1=\alpha_2=1$.
However we do not have a general solution to his question.

\subsection*{Acknowledgements}
We would like to thank Spencer Bloch for asking Golyshev's question. 
This work is supported by JSPS Grant-in-Aid for Scientific Research, 24540001 and 25400007.

\subsection*{Notations}
For $\alpha \in \C$ and an integer $n\geq 0$, $(\alpha)_n=\prod_{i=0}^{n-1}(\alpha+i)$ is the Pochhammer symbol and the generalized hypergeometric function is defined by
\[
{}_pF_{p-1}\left({\alpha_1,\dots,\alpha_p \atop \beta_1,\dots,\beta_{p-1}};x\right)
=\sum_{n=0}^\infty 
\frac{\prod_{i=1}^p(\alpha_i)_n}{\prod_{j=1}^{p-1}(\beta_j)_n} \frac{x^n}{n!}.
\]
When $p=2$, this is called the Gauss hypergeometric function.
We use the standard notation for the product of values of the gamma function $\Gamma(s)$
\[
\Gamma\left({\alpha_1,\dots,\alpha_p 
\atop \beta_1,\dots, \beta_q}\right)
=\frac{\prod_{i=1}^p \Gamma(\alpha_i)}{\prod_{j=1}^q \Gamma(\beta_j)}.
\]

Throughout this paper, we fix an embedding $\ol\Q\hra\C$, and think $\ol\Q$ of 
being a subfield.
For a variety $X$ over $\ol\Q$, $H_\dR^n(X)=H_\dR^n(X/\ol\Q)$ denotes the algebraic de Rham cohomology and $H^n(X,\Q)$ denotes the Betti cohomology of the analytic manifold $X^{an}=(X\times_{\ol\Q}\C)^{an}$.

\section{Betti-de Rham Structures, Hodge-de Rham Structures and Periods}
\label{HdR-sect}
\subsection{Betti-de Rham structures and Hodge-de Rham structures}
\label{HdR-sect1}
Let $k_B,k_\dR$ be fields with fixed embeddings $k_B\hra\C$ and $k_\dR\hra \C$.
A {\it Betti-de Rham structure} over $(k_B,k_\dR)$ (abbreviated BdR) is a datum
$(H_B, H_\dR, \iota)$ consisting of 
\begin{itemize}
\item
a finite dimensional vector space $H_B$ (resp. $H_\dR$) over $k_B$ (resp. $k_\dR$),
\item
a comparison isomorphism
$\iota\colon \C \ot_{k_\dR} H_\dR \os{\sim}{\to} \C\ot_{k_B} H_B$.
\end{itemize}
A {\it Hodge-de Rham structure} over $k_\dR$ (abbreviated HdR) is a datum
$(H_B, H_\dR, F^\bullet,\iota)$ consisting of 
\begin{itemize}
\item
a finite dimensional vector space $H_B$ (resp. $H_\dR$) over $\Q$ (resp. $k_\dR$),
\item
a finite decreasing filtration $F^\bullet$ on $H_\dR$
\item a comparison isomorphism
$\iota\colon \C \ot_{k_\dR} H_\dR \os{\sim}{\to} \C\ot_{k_B} H_B$
\end{itemize}
such that
$(H_B, \C\ot_{k_\dR}H_\dR, \C\ot_{k_\dR}F^\bullet,\iota)$ is a Hodge structure in the usual
sense.
A {\it mixed Hodge-de Rham structure} 
$(H_B, W_B,H_\dR, F^\bullet,W_\dR,\iota)$ over $k_\dR$ (abbreviated MHdR)
is defined in the similar way
where $W_B$ (resp. $W_\dR$) is a finite increasing filtration on $H_B$ (resp. $H_\dR$).
The Tate twists $\Q(r)=(\Q,k_\dR,F^\bullet,\iota)$ is defined as 
$F^{-r}k_\dR=k_\dR$, $F^{-r+1}k_\dR=0$ and the comparison $\iota:k_\dR\to \C$
given by $1\mapsto (2\pi i)^{-r}$.  
The {\it dual} and {\it tensor products} 
of BdR, HdR and MHdR are defined in the customary way.


In this paper we usually consider the case $k_\dR=\ol\Q\hra\C$ with the fixed embedding. 


A {\it filtered Betti-de Rham structure} over $(k_B,k_\dR)$ is a datum
$(H_B, H_\dR, F^\bullet,\iota)$ consisting of  a Betti-de Rham structure
$(H_B, H_\dR, \iota)$
and a finite decreasing filtration $F^\bullet$ on $H_\dR$.
The category $\FiltBdR=\FiltBdR_{k_B,k_\dR}$ 
of filtered Betti-de Rham structures over $(k_B,k_\dR)$ is {\it not} abelian
but exact category. The Yoneda extension groups
\[
\Ext_\FiltBdR^\bullet(H,H')
\]
are defined in the canonical way (cf. \cite{BBDG} 1.1).
The following isomorphism is well-known (cf. \cite{a-o} Proposition 2.1).
\begin{prop}[Carlson's isomorphism]\label{HdR-prop1}
Let $H=(H_B, H_\dR, F^\bullet,\iota)$ be a filtered Betti-de Rham structure.
Then there is a natural isomorphism
$$\Ext_\FiltBdR^1(k,H)\cong (\C\ot_{k_\dR}H_\dR)/(F^0H_\dR+\iota^{-1}H_B)$$
where $k=(k_B,k_\dR,F^\bullet,\id)$ denotes the unit object 
which is defined as $F^0k_\dR=k_\dR$, $F^1k_\dR=0$ and the comparison is the identity.
\end{prop}

\subsection{Periods}\label{HdR-sect2}
For a Betti-de Rham structure
$H=(H_B, H_\dR, \iota)$, the {\it period matrix} of $H$ is defined to be
the representation matrix of $\iota$ with respect to the $k_B$, $k_\dR$-lattices
$H_B$, $H_\dR$, and we denote by
\[
\Per(H)\in \GL_r(k_B)\backslash \GL_r(\C)/\GL_r(k_\dR).
\]

\subsection{Multiplication}\label{HdR-sect3}
A {\it multiplication} on a Betti-de Rham structure $H$ by a commutative $\Q$-algebra $R$
is defined as a ring homomorphism $R\to \End_\BdR(H)$ to the endomorphism ring
of Betti-de Rham structures.
The {\it tensor product} $H_1\ot_R H_2$ over $R$ is defined to be
 \[H_1\ot_R H_2=(H_{1,B}\ot_{k_B\ot R}H_{2,B},H_{1,\dR}\ot_{k_\dR\ot R}H_{2,\dR},\iota_1\ot\iota_2)\]
endowed with multiplication by $R$. 
The multiplication on the dual Betti-de Rham structure
\[
H^*=(\Hom_{k_B}(H_B,k_B),\Hom_{k_\dR}(H_\dR,k_\dR),\iota)
\]
is defined in such a way that $r\phi:=\phi\circ r$ for $\phi\in\Hom(H_B,k_B)$ and 
$r\in R$.

A multiplication on a filtered BdR, HdR, MHdR and 
its $\chi$-parts are defined in the same way as above.
 
 
Assume $\Image(k_B\hra\C)\subset \ol\Q$ and $\Image(k_\dR\hra\C)\subset \ol\Q$
(note that $\ol\Q\hra\C$ is fixed throughout the paper).
For a homomorphism $\chi:R\to\ol\Q$, we define
the {\it $\chi$-part} of a BdR structure $H$ as
\[
H(\chi):=(H_B(\chi),H_\dR(\chi),\iota)
\]
\[
H_B(\chi):=\ol\Q\ot_{k_B\ot R}H_B,\quad
H_\dR(\chi):=\ol\Q\ot_{k_\dR\ot R}H_\dR,
\]
where $k_B\ot R\to \ol\Q$ and
$k_\dR\ot R\to \ol\Q$ are induced from $\chi$ and the embeddings $k_B\hra\ol\Q$
and $k_\dR\hra\ol\Q$.
Then $H(\chi)$ is a BdR over $(\ol\Q,\ol\Q)$.
We call its period matrix
\begin{equation*}\label{HdR-eq1}
\Per(H(\chi))\in\GL_r(\ol\Q)\backslash\GL_r(\C)/\GL_r(\ol\Q)
\end{equation*}
the {\it $\chi$-part of the period matrix} of $H$.
The $\chi$-part of filtered BdR, HdR and MHdR are defined in the same way.


Suppose that $R$ is a semisimple and finite-dimensional $\Q$-algebra. 
Then the functor $\FiltBdR_{k_B,k_\dR}\to \FiltBdR_{\ol\Q,\ol\Q}$ given by 
$H\to H(\chi)$ is exact.
Composing with the forgetting functor $\MHdR_{k_\dR}\to
\FiltBdR_{\Q,k_\dR}$
one has a map
\begin{equation}\label{HdR-eq2}
\Ext^1_{\MHdR_{k_\dR}}(\Q,H)\lra
\Ext^1_{\FiltBdR_{\ol\Q.\ol\Q}}(\ol\Q,H(\chi)),\quad M\longmapsto M(\chi)
\end{equation}
and we call $M(\chi)$ the {\it $\chi$-part of extension class} $M$.

Let $X$ be a smooth projective variety over $k_\dR$.
Let
\[
\reg:H_\cM^i(X,\Q(j))\lra
\Ext^1_\MHdR(\Q,H^{i-1}(X,\Q(j))),\quad i\ne2j
\]
be the Beilinson regulator map.
Suppose that the mixed Hodge de Rham structure $H:=H^{i-1}(X,\Q(j))$ over $k_\dR$
has a multiplication by $R$.
Then we call the composition
\begin{equation*}\label{HdR-eq3}
\reg(\chi):H_\cM^i(X,\Q(j))\lra
\Ext^1_\MHdR(\Q,H)\os{\eqref{HdR-eq2}}{\lra}
\Ext^1_{\FiltBdR_{\ol\Q,\ol\Q}}(\ol\Q,H(\chi))
\end{equation*}
the {\it $\chi$-part of regulator map}. 
\subsection{Variations of Hodge-de Rham structures}
Let $S$ be a smooth variety over $k_\dR$.
A {\it filtered Betti-de Rham structure on $S$} consists of 
a datum $(H_B,H_\dR,F^\bullet, \nabla,\iota)$ where 
\begin{itemize}
\item
$H_B$ is a local system of finite dimensional $k_B$-modules on $S^{an}$,
\item
$H_\dR$ is a locally free $\O_S$-module of finite rank, and $F^\bullet$ is a finite
decreasing filtration which is locally a direct summand,
\item
$(H_\dR,\nabla)$ is a connection with regular singularities on $S$ such that $\nabla(F^p)
\subset \Omega^1_S\ot F^{p-1}$,
\item
$\iota\colon \O_S^{an} \ot_{a^{-1}\O_S} a^{-1}
H_\dR \os{\sim}{\to} \O_S^{an}\ot_{k_B} H_B$ is a comparison isomorphism such that $\nabla$ annihilates the lattice $H_B$, where $a:S^{an}\to S^\zar$ is the canonical map
from the analytic site to the Zariski site and $\O_S^{an}$ denotes the sheaf of analytic functions on $S^{an}$.
\end{itemize}
A filtered BdR on $S$ is called 
a {\it variation of Hodge-de Rham structure} (abbreviated VHdR) if
$k_B=\Q$ and $(H_B, \O^{an}_S\ot_{\O_S}H_\dR, \O^{an}_S\ot_{\O_S}
F^\bullet,\iota,\nabla)$ is a variation of Hodge structure in the usual sense.
We also define, in a customary way, 
a {\it variation of mixed Hodge-de Rham structure} (abbreviated VMHdR) on $S$ 
which consists of
a datum $(H_B, W_\bullet^B,H_\dR, F^\bullet,W_\bullet^\dR,\iota,\nabla)$.

\section{Hypergeometric Fibrations}\label{HG-sect}
In what follows we work over the base field $k_\dR=\ol\Q$.
\subsection{Definition}\label{HG-sect1}
Let $R$ be a finite-dimensional semisimple $\Q$-algebra.
Let $e:R\to E$ be a surjection onto a number field $E$.
Let $X$ be a smooth projective variety over $k_\dR$, and $f:X\to \P^1$
a surjective map.
We say $f$ is a {\it hypergeometric fibration with multiplication by $(R,e)$} if
it is endowed with a multiplication on $R^1f_*\Q|_U$ by $R$
where $U\subset \P^1$ is the maximal Zariski open set such that $f$ is smooth over $U$
and the following conditions hold.
We fix an inhomogeneous coordinate $t\in\P^1$.
\begin{itemize}
\item
$f$ is smooth over $\P^1\setminus\{0,1,\infty\}$,
\item
$\dim_E (R^1f_*\Q)(e)=2$ where we write $V(e):=E\ot_{e,R}V$ the $e$-part,
\item
Let $\Pic_f^0\to \P^1\setminus\{0,1,\infty\}$ be the Picard fibration whose general fiber
is the Picard variety $\Pic^0(f^{-1}(t))$,
and let $\Pic_f^0(e)$ be the component associated to the $e$-part $(R^1f_*\Q)(e)$
(this is well-defined up to isogeny).
Then $\Pic_f^0(e)\to\P^1\setminus\{0,1,\infty\}$ has a totally degenerate semistable reduction at $t=1$.
\end{itemize}  
The last condition is equivalent to say that
the local monodromy $T$ on $(R^1f_*\Q)(e)$ at $t=1$ is unipotent
and the rank of log monodromy $N:=\log(T)$ is maximal, namely 
$\rank(N)=\frac{1}{2}\dim_\Q (R^1f_*\Q)(e)$
($=[E:\Q]$ by the second condition).
\begin{exmp}[Gauss type]
Let $f:X\to \P^1$ be the fibration over $\ol\Q$
whose general fiber is defined by
an affine equation 
$$y^N=x^a(1-x)^b(t-x)^{N-b}$$ 
with $0<a,b<N$ and $\gcd(N,a,b)=1$. 
Let $\mu_N\subset \ol\Q$ be the group of $N$th roots of unity.
It gives automorphisms $(x,y,t)\mapsto (x,\zeta_Ny,t)$ for $\zeta_N\in \mu_N$, 
and then it defines a multiplication by $R=\Q[\mu_N]$ (=group ring) on $R^1f_*\Q$. 
Then one can show that $f$ is a hypergeometric fibration with multiplication by $(R,e)$
if and only $d:=\sharp\ker[e:\mu_N\to E^\times]$ satisfies $ad/N, \,bd/N\not\in\Z$. 
\end{exmp}

\begin{exmp}[Fermat type]
Let $f:X\to \P^1$ be the fibration over $\ol\Q$ defined by
an affine equation 
$$(x^n-1)(y^m-1)=1-t.$$
The group ring $R=\Q[\mu_n,\mu_m]$ acts
on $R^1f_*\Q$ in a natural way.
Then $f$ is a hypergeometric fibration with multiplication by $(R,e)$
if and only if $e:\Q[\mu_n\times\mu_m]\to E$ does not factor through the
projections $\mu_n\times\mu_m\to\mu_n$ nor $\mu_n\times\mu_m\to\mu_m$.
The reason why we call this ``Fermat type'' is the following.
Letting $u=x^{-1}$, $v=y^{-1}$ and $s=tx^{-n}y^{-m}$, 
\[
(x^n-1)(y^m-1)=1-t\quad \Longleftrightarrow \quad u^n+v^m=1+s. 
\] 
\end{exmp}

\subsection{Basic properties}\label{HG-sect2}
We sum up some properties on our hypergeometric fibrations, which will be used in
later sections.
See \cite{a-o} \S 3.3
for complete proofs\footnote{The definition of ``hypergeometric fibrations''
in \cite{a-o} \S 3.1
is slightly different from that in \S \ref{HG-sect1}.
However, the same arguments entirely work in our situation.}.
 
\begin{prop}\label{HG-prop0}
$\dim_{\ol\Q}F^1H^1_\dR(X_t)(\chi)=1$ and
$\dim_{\ol\Q}\Gr_F^0H^1_\dR(X_t)(\chi)=1$
where $X_t$ is a general fiber.
\end{prop}
 \begin{prop}\label{HG-prop1}
 $(R^1f_*\Q)(\chi)$ is an irreducible $\ol\Q[\pi_1(\P^1\setminus\{0,1,\infty\})]$-module.
 Moreover let $\a_0^\chi,\b_0^\chi$ (resp. 
 $\a^\chi,\b^\chi$) be rational numbers such that
$e^{2\pi i\a_0^\chi}$, $e^{2\pi i\b_0^\chi}$
(resp. $e^{2\pi i\a^\chi}$, $e^{2\pi i\b^\chi}$) are eigenvalues of
the local monodromy on $(R^1f_*\Q)(\chi)$ at $t=0$ (resp. $t=\infty$).
Then, none of $\a_0^\chi+\a^\chi$, $\a_0^\chi+\b^\chi$, $\b_0^\chi+\a^\chi$, $\b_0^\chi+\b^\chi$ is an integer.
  \end{prop}
 \begin{prop}\label{HG-prop2}
Let $\psi_{t=1}$ denote the nearby cohomology functor at $t=1$ and let
$W_\bullet$ be the weight monodromy filtration induced by the log monodromy
$N_1=\log(T_1)$.
Then there are isomorphisms
\[
\Gr^W_2\psi_{t=1}R^1f_*\Q(e)=\Coker(N_1)\cong E\ot \Q(-2),\quad
\Gr^W_0\psi_{t=1}R^1f_*\Q(e)=\ker(N_1)\cong E,
\]
\[
\Gr^W_j\psi_{t=1}R^1f_*\Q(e)=0,\quad j\ne 0,2
\]
of Hodge-de Rham structure with compatible $E$-action,
where $E$ is endowed with a trivial Hodge-de Rham structure
of type $(0,0)$.
 \end{prop}
 
\section{Period Formula}\label{period-sect}
\subsection{Setting}\label{period-sect1}
Let $R_0$ be a finite-dimensional semisimple $\Q$-algebra and $e_0:R_0\to E_0$ be 
a surjection onto a number field $E_0$.
Let $f:X\to\P^1$ be a hypergeometric fibration over $\ol\Q$
with multiplication by $(R_0,e_0)$ 
in the sense of
\S  \ref{HG-sect1}.

Let $S:=\bA^1\setminus\{0,1\}$ be defined over $\ol\Q$
with coordinate $\lambda$.
Write $\P^1_S:=\P^1\times S$.
Put $\cU:=(\bA^1\setminus\{0,1\}\times S)\setminus\Delta$
where $\Delta$ is the diagonal subscheme.
Let $l\geq 1$ be an integer. 
Let $\pi:\P^1_S\to\P^1_S$ be a morphism over $S$ given by $(s,\lambda)\mapsto 
(\lambda-s^l,\lambda)$.
Then we consider a variation of Hodge-de Rham structures
\[
\cM:=\pi_*\Q\ot \pr_1^*R^1f_*\Q|_{\mathscr U},\quad\pr_1:
\P^1_S=\P^1\times S\to\P^1
\]
on $\cU$ and a variation of mixed Hodge-de Rham structures
\[
\cH:=R^1\pr_{2*}\cM=(\cH_B,W^B_\bullet,\cH_\dR,F^\bullet,W^\dR_\bullet,\nabla,\iota),
\quad \pr_2:\cU\to S
\] 
on $S$.
Since $\cM$ is a variation of Hodge-de Rham structures
of pure weight $1$, the weights of $\cH$ are at most $2,3$ and $4$. 
We have an exact sequence
\begin{equation}\label{period-eq1}
0\lra W_2\cH\lra\cH\lra\cH/W_2\cH\lra0
\end{equation}
with $W_2\cH$ a variation of Hodge-de Rham structures of pure weight 2.

The group $\mu_l\subset \ol\Q^\times$ of $l$th roots of unity acts on 
the cyclic covering $\pi$.
Thus the group ring $R:=R_0[\mu_l]$ acts
on the sheaf $\cM$ and hence on $\cH$.
In what follows we fix $e:R\to E$ a surjection onto a number field $E$ such that
$\ker(e)\supset \ker(e_0)$.
There is a unique embedding $E_0\hra E$ making the diagram
\[
\xymatrix{
R_0\ar[r]^{e_0}\ar[d]&E_0\ar[d]\\
R\ar[r]^e&E
}
\]
commutative.
Then the $e$-part 
$$\cM(e):=E\ot_{e,R}\cM\cong E\ot_{E_0[\mu_l]}(\pi_*\Q\ot R^1f_*\Q(e_0))$$ 
is of rank 2 over $E$.

Let $\chi:R\to \ol\Q$ be a homomorphism factoring through $e$.
This also induces $R_0\to \ol\Q$ which we also write $\chi$ by abuse of notation.
Define $k\in\{0,1,\dots, l-1\}$ by $\chi(\zeta_l)=\zeta_l^k$ for $\forall\zeta_l\in\mu_l$.
Put $q^\chi:=k/l$.
Moreover, let $\a_0^\chi$, $\b_0^\chi$ (resp.
$\a^\chi$, $\b^\chi$) be rational numbers in the interval $[0,1)$ 
such that
$e^{2\pi i\a_0^\chi}$ and $e^{2\pi i\b_0^\chi}$
(resp. $e^{2\pi i\a^\chi}$ and $e^{2\pi i\b^\chi}$) are eigenvalues of
the local monodromy $T_0$ at $t=0$ (resp. $T_\infty$ at $t=\infty$)
on the $\chi$-part $(R^1f_*\Q)(\chi)=\ol\Q\ot_{\chi,R_0}R^1f_*\Q$.
Equivalently, these are congruent mod $\Z$ to 
the eigenvalues of the residue $\Res_{t=0}(\nabla)$ (resp. $\Res_{t=\infty}(\nabla)$) 
of the connection on the $\chi$-part of the bundle 
$R^1f_*\Omega^\bullet_{X/\P^1}|_{\P^1\setminus\{0,1,\infty\}}$.
Note that the local monodromy $T_1$ at $t=1$ is unipotent.
Since $T_0T_1T_\infty$ is the identity, we have
\[
\a_0^\chi+\b_0^\chi+\a^\chi+\b^\chi \in\Z.
\] 
\subsection{Theorem on periods}\label{period-sect2}
Let $\O^{an}$ be the sheaf of analytic functions on $S^{an}$,  
$\O^\zar$ the Zariski sheaf of rational functions (with coefficients in $\ol\Q$) on 
$S$ with coordinate $\lambda$.
Let $a$
be the canonical morphism from the analytic site to the Zariski site.
We put
\[
W_2\cH^{an}_\dR:=\O^{an}\ot_{a^{-1}\O^\zar}a^{-1}W_2\cH_\dR,\quad
W_2\cH_B^{an}:=\O^{an}\ot_\Q W_2\cH_B
\]
sheaves on the analytic site.
The comparison isomorphism of $W_2\cH$ gives an analytic section
\[
\iota\in \vg(S^{an},
\mathit{Isom}(W_2\cH_\dR^{an},W_2\cH_B^{an})).
\]
There is a canonical map ($d:=\rank W_2\cH$)
\[
\mathit{Isom}(W_2\cH_\dR^{an},W_2\cH_B^{an})
\lra\GL_d(\Q)\backslash
\GL_d(\O^{an})
/\GL_d(a^{-1}\O^\zar)
\]
of sheaves by associating the representation matrices with respect to the lattices
$W_2\cH_B$ and $W_2\cH_\dR$.
We call the image of $\iota$ the {\it period matrix} of $W_2\cH$ :
\[
\Per(W_2\cH)\in \vg(S^{an},
\GL_d(\Q)\backslash
\GL_d(\O^{an})
/\GL_d(a^{-1}\O^\zar)).
\]
The $\chi$-part of $W_2\cH$ defines the $\chi$-part of the period matrix  ($r:=\rank W_2\cH(\chi)$)
\[
\Per(W_2\cH(\chi))\in \vg(S^{an},
\GL_r(\ol\Q)\backslash
\GL_r(\O^{an})
/\GL_r(a^{-1}\O^\zar)).
\]
\begin{thm}[Period formula]\label{main1}
Assume $\a_0^\chi=0$ $(\Leftrightarrow$ $\alpha^\chi_0+\alpha^\chi+\beta^\chi\in\Z)$
and that
$q^\chi \not \equiv 0$, $\a^\chi$, $\b^\chi$, $\a^\chi+\b^\chi \pmod \Z$. 
Then the rank of $W_2\cH(\chi)$ is $2$. 
For some $\m>1$, $\m\equiv q^\chi \pmod\Z$, 
the period matrix is
locally given by
$$\Per(W_2\cH(\chi))=
2\pi i\begin{pmatrix}
\Theta F_\m(\l) &\Theta G_\m(\l) \\ \pl \Theta F_\m(\l) & \pl \Theta G_\m(\l)
\end{pmatrix},$$
where we put
\begin{align*}
F_\m(\l)&:=\frac{1}{\m}(\l-1)^\m\FF{\a^\chi,\b^\chi}{\m+1}{1-\l}, 
\\ G_\m(\l)&:=(-1)^\m \GG{\m, \m+1-\a^\chi-\b^\chi}{\m+1-\a^\chi,\m+1-\b^\chi}\FF{\a^\chi-\m,\b^\chi-\m}{\a^\chi+\b^\chi-\m}{\l},
\end{align*}
and $\Theta$ is a differential operator of the form
$$\Theta = q(\l)+r(\l)\pl, \quad q(\l), r(\l) \in \ol\Q[\l,1/\l(\l-1)]. $$
\end{thm}

\subsection{Proof of Period formula: Part 1}\label{period-sect3}
We first show $\dim_{\ol\Q}W_2\cH(\chi)=2$.
We write $U_\lam:=\pr_2^{-1}(\lam)\cong
\P^1\setminus\{0,1,\lam,\infty\}$ for $\lam\in S^{an}=\C\setminus\{0,1\}$ where $\pr_2:\cU\to S$.
Moreover we put the fibers
\[
\cM_\lam:=\cM|_{\pr_2^{-1}(\lam)}
\cong \pi_{\lam*}\Q\ot R^1f_*\Q,
\quad 
H_\lam:=\cH|_{\{\lam\}}\cong H^1(\pr_2^{-1}(\lam),\cM_\lam), 
\]
where $\pi_\lam:\P^1\to\P^1$ is the map given by $s\mapsto \lam-s^l$.
When $\lam\in\ol\Q\setminus\{0,1\}$, we endow $\cM_\lam$ and $\cH_\lam$ with HdR structure over $\ol\Q$
induced from the
$\ol\Q$-frames on $\cM_\dR$ and $\cH_\dR$ respectively.
We then want to show $\dim_{\ol\Q}W_2H_\lam(\chi)=2$ for $\lam\in\ol\Q\setminus\{0,1\}$.
The weight filtration induces an exact sequence
\[
0\lra W_2H_\lam(\chi)\lra H_\lam(\chi)\lra H_\lam(\chi)/W_2H_\lam(\chi) \lra0.
\]
There are canonical isomorphisms
\begin{align}
& W_2H_\lam\cong H^1(\P^1,j_*\cM_\lam), \quad j:\P^1\setminus\{0,1,\lam,\infty\}
\hra\P^1,\label{period-eq2-1}
\\& H_\lam/W_2H_\lam\cong H^0(\P^1,R^1j_*\cM_\lam)\label{period-eq2-2}
\end{align}
of mixed Hodge-de Rham structures.
Let $\varepsilon_k:\Q[\mu_l]\to\ol\Q$ be given by $\varepsilon(\zeta_l)=\zeta_l^k$
and $\ol\Q(\varepsilon_k):=\ol\Q\ot_{\varepsilon_k,\Q[\mu_l]}\pi_*\Q$
a one-dimensional local system on $\P^1\setminus\{\lam,\infty\}$.
Then there is a natural isomorphism
\[
\cM_\lam(\chi)\cong \ol\Q(\varepsilon_k)\ot_{\ol\Q} (R^1f_*\Q)(\chi)
\]
of $\pi_1(\P^1\setminus\{0,1,\lam,\infty\})$-modules.
Since $(R^1f_*\Q)(\chi)$ is an irreducible $\pi_1(\P^1\setminus\{0,1,\infty\})$-module
(Proposition \ref{HG-prop1}),
so is $\cM_\lam(\chi)$ as a $\pi_1(\P^1\setminus\{0,1,\lam,\infty\})$-module.
In particular $H^0(p^{-1}(\lam),\cM_\lam(\chi))=0$.
Hence
\begin{align*}
\dim_{\ol\Q}H_\lam(\chi)&=
\dim_{\ol\Q}H^1(p^{-1}(\lam),\cM_\lam(\chi))
=-\chi(p^{-1}(\lam),\cM_\lam(\chi))\\
&=-\chi^\top(p^{-1}(\lam))\cdot\dim_{\ol\Q}\cM_\lam(\chi)
=-(-2)\cdot2=4.
\end{align*}
Thus $\dim_{\ol\Q}W_2H_\lam(\chi)=2$ $\Leftrightarrow$
$\dim_{\ol\Q}H_\lam(\chi)/W_2=2$.


Let $T_0,T_1,T_\lam,T_\infty$ be
the local monodromy on $\cM_\lam(\chi)$ at 
$t=0,1,\lam,\infty$, respectively.
$T_1$ is unipotent with trivial action on $\ol\Q(\varepsilon_k)$, and
$T_\lam$ is multiplication by $e^{2\pi i q^\chi}$ with trivial action on $(R^1f_*\Q)(\chi)$. 
The eigenvalues of $T_0$ (resp. $T_\infty$) are 
\[
e^{2\pi i\a_0^\chi},~e^{2\pi i\b_0^\chi}\quad
(\mbox{resp. }e^{2\pi i(-q^\chi+\a^\chi)},~e^{2\pi i(-q^\chi+\b^\chi)}).
\]
Recall \eqref{period-eq2-2}.
We then have
\begin{align*}
H_\lam/W_2&\cong H^0(\P^1,R^1j_*\cM_\lam)\\
&\cong\bigoplus_{p=0,1,\lam,\infty}
\Coker[T_p-1:\psi_{t=p}\cM_\lam\to \psi_{t=p}\cM_\lam\ot\Q(-1)]
\end{align*}
where $\psi_{t=p}$ denotes the nearby cohomology at $t=p$.
By the assumption $q^\chi \not\equiv 0$, $\a^\chi$, $\b^\chi \pmod\Z$, 
$T_\lam$ and $T_\infty$ have no eigenvalue $1$ 
on the $\chi$-part of $\psi\cM_\lam$.
Hence $\Coker[T_\lam-1]=\Coker[T_\infty-1]=0$.
Note
\begin{align*}
&\Coker[T_p-1:\psi_{t=p}\cM_\lam(\chi)
\to \psi_{t=p}\cM_\lam(\chi)\ot\Q(-1)]\\
&\cong 
\ol\Q(\varepsilon_k)\ot_{\ol\Q}\Coker[T_p-1:(R^1f_*\Q)(\chi)
\to (R^1f_*\Q)(\chi)],\quad p=0,1.
\end{align*}
It follows from Proposition \ref{HG-prop2} that we have $\dim_{\ol\Q}\Coker(T_1-1)=1$.
There remains to show $\dim_{\ol\Q}\Coker(T_0-1)=1$.
By the assumption $\alpha_0^\chi=0$,
one has $\dim_{\ol\Q}\Coker(T_0-1)\geq 1$.
If $\dim_{\ol\Q}\Coker(T_0-1)=2$ this means that $T_0$ is trivial on $(R^1f_*\Q)(\chi)$.
Then $(R^1f_*\Q)(\chi)$ cannot be irreducible as 
$\ol\Q[\pi_1(\P^1\setminus\{0,1,\infty\})]$-module.
This contradicts with Proposition \ref{HG-prop1}. 
We thus have $\dim_{\ol\Q}\Coker(T_0-1)=1$, 
and hence $\dim_{\ol\Q}H_\lam(\chi)/W_2=2$.
This completes the proof of 
$\dim_{\ol\Q}W_2H_\lam(\chi)=2$.


In the discussion above, we obtained the following.
\begin{prop}\label{HG-prop3}
Assume $\alpha_0^\chi\in\Z$ and $q^\chi\not\equiv 0$, $\a^\chi$, $\b^\chi \pmod \Z$.
Then there is an isomorphism
\[
\cH(e)/W_2\cong \bigoplus_{p=0,1}
\Coker[T_p-1:\psi_{t=p}\cM(e)
\to \psi_{t=p}\cM(e)\ot\Q(-1)]
\]
of variations of mixed Hodge-de Rham structures on $\P^1\setminus\{0,1,\infty\}$.
Each $\Coker(T_p-1)$ is one-dimensional over $E$.
Moreover, $\Coker(T_1-1)$ is endowed with a
Hodge-de Rham structure of type $(2,2)$ (Proposition \ref{HG-prop2}).
\end{prop}
\subsection{Relative 1-form $\omega^\chi$}\label{period-sect4}
The $\chi$-part $(f_*\Omega^1_{X/\P^1})(\chi)|_{\P^1\setminus\{0,1,\infty\}}$ 
of the Hodge bundle has rank one (Proposition \ref{HG-prop0}).
Hence it is a trivial line bindle on $\P^1\setminus\{0,1,\infty\}$.
In what follows
we fix a relative $1$-form
\[
\omega^\chi\in \vg(\P^1\setminus\{0,1,\infty\},(f_*\Omega^1_{X/\P^1})(\chi))
\]
with coefficients in $\ol\Q$
which is everywhere nonzero (until the end of the paper).
Let $X_t=f^{-1}(t)$ be the general fiber.
We fix (nonzero) homology cycles
\[
\gamma=\gamma_t\in H_1(X_t,\ol\Q)(\chi)\cap\ker(T_0-1),\quad
\delta=\delta_t\in H_1(X_t,\ol\Q)(\chi)\cap\ker(T_1-1).
\]
Note that each 
$H_1(X_t,\ol\Q)(\chi)\cap\ker(T_p-1)$ is one-dimensional over $\ol\Q$ 
(Proposition \ref{HG-prop3}).
\begin{lem}[Key Lemma]\label{key-lem1}
There is a differential operator $\theta=p_0(t)+p_1(t)\frac{d}{dt}$ with
$p_i(t)\in\ol\Q[t]$ such that
\begin{equation}\label{key-lem1-eq}
\int_\gamma\omega^\chi=B(\a^\chi,\b^\chi)\cdot\theta \FF{\a^\chi,\b^\chi}{\a^\chi+\b^\chi}{t}, 
\quad
\int_\delta\omega^\chi=2\pi i\cdot \theta \FF{\a^\chi,\b^\chi}{1}{1-t}. 
\end{equation}
Here $B(\a,\b)$ is the beta function. 
\end{lem}

\begin{proof}
Since $\delta$ is $T_1$-invariant,
$\int_\delta\omega^\chi$ is a single-valued meromorphic function at $t=1$.
So is $\int_\gamma \omega^\chi$ at $t=0$ as $\gamma$ is $T_0$-invariant.
Therefore \eqref{key-lem1-eq} follows from \cite{a-o}, Lemmas 5.2, 5.3 and 5.4.
\end{proof}

The differential equation
$$(t(D+\a)(D+\b)-(D+\a+\b-1)D)f=0, \quad D=t\frac{d}{dt}$$
has the Riemann scheme
$$\left\{\begin{matrix}t=0 & t=1 & t=\infty \\ 
0 & 0 & \a\\
1-\a-\b&0&\b
\end{matrix}\right\}. $$
Among  Kummer's $24$ solutions, we used in the preceding lemma 
\begin{align}\label{f_1 f_2}
f_1(t):=\FF{\a,\b}{\a+\b}{t},
\quad f_2(t):=\FF{\a,\b}{1}{1-t}. 
\end{align}
Later in Section 5.4, we will use the solution
\begin{equation}\label{f_3}
f_3(t):=t^{1-\a-\b}\FF{1-\a,1-\b}{2-\a-\b}{t},
\end{equation}
which has the characteristic exponent $1-\a-\b$ at $t=0$. 
These satisfy the linear relation (cf. \cite[\S2.9]{bateman})
\begin{equation*}\label{relation f}
\GG{1-\a-\b}{1-\a,1-\b} f_1(t)-f_2(t)+\GG{\a+\b-1}{\a,\b}f_3(t)=0.
\end{equation*}
This can be written, using functional equations of the gamma function
$$B(\a,\b)=\GG{\a,\b}{\a+\b}, \quad \Gamma(s+1)=s\Gamma(s), $$
as
\begin{equation}\label{f123}
B(\a,\b)f_1(t) +2\pi i \frac{1-e^{2\pi i(\a+\b)}}{(1-e^{2\pi i\a})(1-e^{2\pi i\b})}f_2(t)-B(1-\a,1-\b)f_3(t)=0.
\end{equation}

\subsection{Rational 2-forms $s^{m-1}ds\wedge\omega^\chi$.}\label{period-sect5}
By taking an embedded resolution, we may assume that the reduced divisor
$$D:=(f^{-1}(0)+f^{-1}(1)+f^{-1}(\infty))_{\mathrm{red}}$$ of the singular fibers
is a NCD.
Recall from Lemma \ref{key-lem1} (Key Lemma) the differential operator
$\theta=p_0(t)+p_1(t)\frac{d}{dt}$.
By replacing $\omega^\chi$ with $t^n(1-t)^m\omega^\chi$ for some $n,m\geq 0$, 
we may assume without loss of generality
\begin{description}
\item[P1]
$p_i(t)$ are polynomials and $t(1-t)|p_1(t)$,
\item[P2]
$\omega^\chi\in \vg(\P^1\setminus\{\infty\},f_*\Omega^1_{X/\P^1}(\log D))$,
where the locally free sheaf $\Omega^1_{X/\P^1}(\log D)$ is defined by the exact sequence
\[
0\lra f^*\Omega^1_{\P^1}(\log(0+1+\infty))
\lra \Omega^1_X(\log D)\lra \Omega^1_{X/\P^1}(\log D)\lra 0.
\]
\end{description}
Let $\lam\in \C\setminus\{0,1\}$ and $\pi_\lam:\P^1\to\P^1$
a morphism given by $s \mapsto\lam-s^l$ as in \S \ref{period-sect3}.
To distinguish the source and target of $\pi_\lam$, 
we denote by $\P^1$ the target with inhomogeneous coordinate $t$, and
by $\P^1_\lam$ the source with inhomogeneous coordinate $s$. 
Let
\[
\xymatrix{
X_\lam\ar[r]^{i\qquad}\ar[rd]_{f_\lam}&\P^1_\lam
\times_{\P^1}X\ar[r]\ar[d]&X\ar[d]^f\\
&\P^1_\lam\ar[r]^{\pi_\lam}&\P^1
}
\]
where $i$ is the desingularization such that the inverse image $D_\lam\subset X_\lam$
of $D$ is a NCD.
Let $U_\lam\subset X_\lam$ 
(resp. $\ol{U}_\lam\subset X_\lam$) be the inverse image of 
$\P^1\setminus\{0,1,\lam,\infty\}$
(resp. $\P^1\setminus\{\infty\}$) under $\pi_\lam \circ f_\lam$.
By the projection formula, 
$$\cM_\lam=\pi_{\lam*}\Q\ot R^1f_*\Q
\cong \pi_{\lam*}(\pi_\lam^*R^1f_*\Q)=\pi_{\lam*}(R^1f_{\lam*}\Q).$$
This implies
\[
H_\lam=
H^1(\P^1\setminus\{0,1,\lam,\infty\},\cM_\lam)\cong 
H^1(\P^1_\lam\setminus\{s^l=0,\lam,\lam-1,\infty\},
R^1f_{\lam*}\Q)
\]
and hence we have an exact sequence
\begin{equation*}\label{period-sect5-eq1}
\xymatrix{
0\ar[r]&
H_\lam\ar[r]
&H^2(U_\lam,\Q)\ar[r]&H^2(f^{-1}_\lam(s),\Q).
}
\end{equation*}
In particular, by taking the weight $2$ piece, we have a canonical isomorphism
\begin{align}\label{period-sect5-eq2}
W_2H_\lam
\os{\cong}{\lra}
&\ker[W_2H^2(U_\lam,\Q)\to H^2(f^{-1}_\lam(s),\Q)]\notag
\\=&\ker[H^2(X_\lam,\Q)/H^2_{D_\lam}(X_\lam)\to H^2(f^{-1}_\lam(s),\Q)].
\end{align}
Consider the rational 2-form
\[
s^{m-1}ds\,\omega^\chi=s^{m-1}ds\wedge\omega^\chi\in \vg(\cU,\Omega^2_{\cU/\ol\Q})
\] for  $m\geq 1$.
By the assumption \textbf{P2},
\begin{align*}
s^{m-1}ds\wedge \omega^\chi|_{\lambda=\lam}&
\in
\Image[\Omega^1_{\P^1_\lam\setminus\{\infty\}}\wedge\Omega^1_{\ol{U}/\P^1}(\log D)\lra\Omega^2_{\ol{U}_\lam}(\log D_\lam)],
\\&\subset
\Image[t(1-t)\Omega^1_{\P^1_\lam}(\log(\pi_\lam^{-1}(0+1)))\wedge\Omega^1_{\ol{U}/\P^1}(\log D)\to\Omega^2_{\ol{U}_\lam}(\log D_\lam)]\\
&\subset t(1-t)\Omega^2_{\ol{U}_\lam}(\log D_\lam)=\Omega^2_{\ol{U}_\lam}, 
\end{align*}
so that we have
\[
s^{m-1}ds\wedge\omega^\chi |_{\lambda=\lam}\in \vg(\ol{U}_\lam,\Omega^2_{\ol{U}_\lam}).
\]
Let $[s^{m-1}ds\wedge\omega^\chi]|_{\lambda=\lam}\in H_\dR^2(\ol{U}_\lam/\C)$ denote the de Rham cohomology class.
Obviously, its restriction to the general fiber $f_\lam^{-1}(s)$ vanishes. Thus
\[
[s^{m-1}ds\wedge\omega^\chi]|_{\lambda=\lam}\in H^1(\P^1\setminus\{0,1,\lam,\infty\},\cM_\lam)\cap
\Image H^2_\dR(\ol{U}_\lam/\C).
\]
If $m\equiv k$ modulo $l$, then $[s^{m-1}ds\wedge\omega^\chi]|_{\lambda=\lam}$ 
belongs to the $\chi$-part. 
By Proposition \ref{HG-prop3} together with the commutative diagram
\[
\xymatrix{
H_\lam\ar[r]\ar[d]&H^2(U_\lam,\Q)\ar[d]\\
H_\lam/W_2\ar[r]^{i\quad }&H^3_{D_\lam}(X_\lam,\Q)
}
\]
with injective $i$, we have
\[
[s^{m-1}ds\wedge\omega^\chi]|_{\lambda=\lam}
\in W_2H_{\lam,\dR}(\chi)=W_2H^1(\P^1\setminus\{0,1,\lam,\infty\},\cM_\lam)(\chi).
\]
This means
\[
[s^{m-1}ds\wedge\omega^\chi]\in \vg(S,W_2\cH(\chi)).
\]
Summarizing, we have: 
\begin{lem}\label{period-lem1}
Let $\omega^\chi$ be a relative 1-form satisfying the condition {\rm\textbf{P2}}.
Then for any integer $m\geq 1$ such that $m\equiv k \pmod l$, the rational $2$-form
$s^{m-1}ds\wedge\omega^\chi \in \vg(\cU,\Omega^2_{\cU/\ol\Q})$ defines
a de Rham cohomology class
\[[s^{m-1}ds\wedge\omega^\chi]\in \vg(S,W_2\cH(\chi)).
\]
\end{lem}
As is shown in \S \ref{period-sect3}, $W_2H_{\lam,\dR}(\chi)$ is two-dimensional.
We shall show in below that it is spanned by
$[s^{m-1}ds\wedge\omega^\chi]|_{\lambda=\lam}$
and $[s^{m-l-1}ds\wedge\omega^\chi]|_{\lambda=\lam}$ for some $m$.
We note that it is never obvious to show even the non-vanishing.
\subsection{Proof of Period formula: Part 2}\label{period-sect6}
We compute the period matrix $\Per(W_2\cH(\chi))$.

Let $A_\dR\subset H^2_\dR(\ol{U}_\lam)$ be the $\ol\Q$-subspace 
spanned by $[s^{m-1}ds\wedge\omega^\chi]|_{\lambda=\lam}$
with $m>0$ such that $m\equiv k$ mod $l$ (cf. Lemma \ref{period-lem1}).
Consider the commutative diagram
\[
\xymatrix{
A_\dR\ar[d]\ar[r]&\Hom(H_2^B(\ol{U}_\lam),\C)\ar@{=}[r]\ar[d]
&H_B^2(\ol{U}_\lam,\C)\ar[d]\\
W_2H_{\lam,\dR}(\chi)\ar[r]&\Hom(H_2^B(U_\lam),\C)\ar@{=}[r]&
H_B^2(U_\lam,\C).
}
\]
As is easily shown, 
\[
H^2(\ol{U}_\lam)_{\mathrm{fib}}:=\ker[H^2(\ol{U}_\lam)
\to H^2(D_\lam\cap\ol{U}_\lam)]\lra
H^2(U_\lam)
\]
is injective (cf. \cite{a-o} \S 6.1), and $A_\dR$ is obviously contained in 
$H^2_\dR(\ol{U}_\lam)_{\mathrm{fib}}$.
Our goal is to find 
$m_i$ and 
$Z_i\in H_2^B(\ol{U}_\lam,\ol\Q)$ ($i=1,2$)
such that
\begin{equation}\label{period-sect7-prop1}
\begin{vmatrix}
\int_{Z_1}s^{m_1-1}ds\wedge\omega^\chi|_{\l=\lam}
&\int_{Z_2}s^{m_1-1}ds\wedge\omega^\chi|_{\l=\lam}\\
\int_{Z_1}s^{m_2-1}ds\wedge\omega^\chi|_{\l=\lam}
&\int_{Z_2}s^{m_2-1}ds\wedge\omega^\chi|_{\l=\lam}
\end{vmatrix}\ne0
\end{equation}
and show that the entries (regarded as analytic functions of variable $\lam$) 
are as in Theorem \ref{main1}.
Then this gives the period matrix $\Per(W_2\cH(\chi))$.

Recall from \S \ref{period-sect4} the homology cycle
$\delta\in H_1(f^{-1}(t),\ol\Q)(\chi)$.
We think it being a homology cycle in a fiber $f^{-1}_\lam(s)$.
We take the Lefschetz thimble $\Delta\subset \ol{U}_\lam$ over a segment from $s=0$
to $s=\sqrt[l]{\lam-1}$ (a fixed $l$th root).
Let $\zeta\in \mu_l$ be
a primitive $l$th root of unity and $\sigma_\zeta\in \Aut(\pi_\lam)$ 
be the corresponding automorphism.
We denote the automorphism of $X_\lam$
induced from $\sigma_\zeta\times \id_X$ by the same symbol $\sigma_\zeta$.
$\Delta$ has no boundary over $s=\sqrt[l]{\lam-1}$, but may have boundary over $s=0$.
Since $\sigma_\zeta$ acts on the fiber over $s=0$ as identity,
$(1-\sigma_\zeta)\Delta$ has no boundary:
\[
(1-\sigma_\zeta)\Delta\in H_2(\ol{U}_\lam,\ol\Q).
\]
Let $T_{\lam=0}$ denote the local monodromy at $\lam=0$ 
on $H_2(\ol{U}_\lam,\ol\Q)$ and we put 
\begin{equation}\label{period-eq4}
Z_1:=(1-\sigma_\zeta)\Delta,\quad Z_2:=T_{\lam=0}(Z_1)\in H_2(\ol{U}_\lam,\ol\Q).
\end{equation}

Let $p_i(t)\in\ol\Q[t]$ be polynomials 
which satisfy {\bf P1} in the beginning of \S \ref{period-sect5}.
Put 
\begin{equation}\label{p to ab}
a_i(\l):=\frac{(-1)^i}{i!}\partial_\l^i p_0(\l),\quad
b_i(\l):=\frac{(-1)^i}{i!} \partial_\l^i p_1(\l), 
\end{equation}
so that
\begin{equation}\label{ab to p}
p_0(t)=\sum_{i=0}^N a_i(\lambda)(\lambda-t)^i,\quad
p_1(t)=\sum_{i=0}^N b_i(\lambda)(\lambda-t)^i, 
\end{equation}
for a sufficiently large $N$. 

\begin{lem}\label{partial F G}
Let $F_\m(\l)$ and $G_m(\l)$ be as defined in Theorem \ref{main1}. 
Then we have, for $\m>1$,
\begin{align}
\partial_\l F_\m(\l)&=(\m-1)F_{\m-1}(\l),\label{partial F}
\\
\pl G_\m(\l)&=(\m-1)G_{\m-1}(\l). \label{partial G}
\end{align}
\end{lem}

\begin{proof}
Write $\a=\a^\chi$, $\b=\b^\chi$. Since 
\begin{equation}\label{partial 2F1}
\partial_\l \FF{\a,\b}{\gamma}{\l}=\frac{\a\b}{\gamma}\FF{\a+1,\b+1}{\gamma+1}{\l}
\end{equation}
in general, we have
$$\partial_\l F_\m(\l) 
= (\l-1)^{\m-1} \left(\FF{\a,\b}{\m+1}{1-\l}+\frac{\a \b}{\m(\m+1)}(1-\l)\FF{\a+1,\b+1}{\m+2}{1-\l}\right).$$
Hence \eqref{partial F} is equivalent to
$$\frac{(\a)_n(\b)_n}{(\m+1)_n(1)_n}+\frac{\a\b}{\m(\m+1)}\frac{(\a+1)_{n-1}(\b+1)_{n-1}}{(\m+2)_{n-1}(1)_{n-1}}=\frac{(\a)_n(\b)_n}{(\m)_n(1)_n} \quad (n \ge 1),$$
and this is easily verified. 
One proves \eqref{partial G} similarly, using \eqref{partial 2F1} and the functional equation $\Gamma(s+1)=s\Gamma(s)$. 
\end{proof}

\begin{prop}\label{prop P_m}
For $\lam\in \C\setminus\{0,1\}$ and any positive integer $m \equiv k \pmod{l}$, put
$$P_m(\lam):= \int_\Delta s^{m-1}ds\wedge\omega^\chi|_{\l=\lam}, \quad \m=\frac{m}{l}$$
and regard it as an analytic function $P_m(\l)$ of variable $\l$.
Then we have
\begin{equation}\label{P_m}
P_m(\l)=\frac{2\pi i}{l} 
\sum_{i=0}^N \left(a_i(\l)+b_i(\l)\partial_\l\right)
F_{\mu+i}(\l). 
\end{equation}
Moreover we have, if $\m>1$, 
\begin{equation}\label{partial P_m}
\partial_\l P_m(\l)=(\m-1)P_{m-l}(\l).
\end{equation}
\end{prop}

\begin{proof}
Letting $t=\l-s^l$ we have by \eqref{key-lem1-eq}
\begin{align*}
P_m(\l)&=\frac{2\pi i}{l}\int_1^\lambda(\lambda-t)^{\m-1}
(p_0(t)+p_1(t)\partial_t) \FF{\a^\chi,\b^\chi}{1}{1-t}\, dt\\
&=\frac{2 \pi i}{l}\int_1^\lambda\left((\lambda-t)^{\m-1}p_0(t)-\partial_t\left((\lambda-t)^{\m-1}p_1(t)\right)\right)
\FF{\a^\chi,\b^\chi}{1}{1-t}\, dt. 
\end{align*}
Here the second equality follows from 
integration by parts and the assumption $1-t\mid p_1(t)$ in {\bf P1}.
By \eqref{ab to p} and letting $1-t=(1-\l)u$, we have
\begin{align*}
P_m(\l)& =\frac{2 \pi i }{l} \sum_{i\geq -1} (a_i(\l)+(\m+i)b_{i+1}(\l))
\int_1^\lambda(\lambda-t)^{\m+i-1}
\FF{\a^\chi,\b^\chi}{1}{1-t}\, dt\\
&=\frac{2 \pi i}{l} \sum_{i\geq -1} (a_i(\l)+(\m+i)b_{i+1}(\l)) 
\\& \hspace{2cm} \times (\l-1)^{\m+i}\int_0^1(1-u)^{\m+i-1}
\FF{\a^\chi,\b^\chi}{1}{(1-\l)u)}\, du.
\end{align*}
By the integral representation of ${}_3F_2$ (cf. \cite{slater}, (4.1.2))
\begin{equation}\label{int rep 3F2}
\int_0^1\FF{a,b}{d}{xt} \,t^{c-1}(1-t)^{e-c-1}\, dt =B(c,e-c){}_3F_2\left({a,b,c\atop d,e};x\right),
\end{equation}
we have
\begin{align*}
&\int_0^1(1-u)^{\m+i-1} \FF{\a^\chi,\b^\chi}{1}{(1-\l)u}\, du
\\&=B(1,\m+i)\FFF{\a^\chi,\b^\chi,1}{1,\m+i+1}{1-\l}
=\frac{1}{\m+i} \FF{\a^\chi,\b^\chi}{\m+i+1}{1-\l}.
\end{align*}
Hence we obtain
$$P_m(\l)=\frac{2\pi i}{l} \sum_{i\geq -1} (a_i(\l)+(\m+i)b_{i+1}(\l))F_{\m+i}(\l).$$
Now \eqref{P_m} follows using \eqref{partial F}, from which \eqref{partial P_m} follows using $$\partial_\l a_i(\l)=-(i+1)a_{i+1}(\l), \quad \partial_\l b_i(\l)=-(i+1)b_{i+1}(\l),$$ 
and \eqref{partial F}. 
\end{proof}

\begin{prop}\label{Theta}
Let the notations be as in Proposition \ref{prop P_m}. 
There is a differential operator $\Theta = q(\l)+r(\l)\pl$ with $q(\l)$, $r(\l) \in \ol\Q[\l,1/\l(\l-1)]$,  such that 
$$P_m(\l)=2 \pi i \cdot \Theta F_\mu(\l).$$
\end{prop}

\begin{proof}
By \eqref{partial F}, 
we have 
$F_{\m+i}(\l)=\frac{(\m)_i}{(\m)_N}\pl^{N-i}F_{\m+N}(\l)$ ($i=0,1,\dots, N$). 
Hence by Proposition \ref{prop P_m}, we have 
$P_m(\l)=2\pi i\cdot \Theta_1 F_{\m+N}(\l)$ with 
$$\Theta_1=\frac{1}{l}\sum_{i=0}^N \left(\frac{(\m)_i}{(\m)_N}(a_i(\l)\pl^{N-i}+b_i(\l)\pl^{N+1-i})\right).$$
Recall that $F_\m(\l)$ is a solution of the differential equation satisfied by
$$\FF{\a^\chi-\m,\b^\chi-\m}{\a^\chi+\b^\chi-\m}{\l},$$
i.e., $\cD F_\m(\l)=0$ with
\begin{equation}\label{DE}
\cD:=\l(1-\l)\pl^2+\{\a^\chi+\b^\chi-\m-(\a^\chi+\b^\chi-2\m+1)\l\}\pl-(\a^\chi-\m)(\b^\chi-\m).
\end{equation}
Hence we have
$$(\a^\chi-\m)(\b^\chi-\m)F_\m=(\{\a^\chi+\b^\chi-\m-(\a^\chi+\b^\chi-2\m+1)\l\}-\l(1-\l)\pl)(\m-1)F_{\m-1}.$$
Applying this iteratively, we obtain a differential operator $\Theta_2$ of degree $N$ such that
$$F_{\m+N}(\l)=\Theta_2 F_\m(\l).$$
By reducing the degree of $\Theta_1\Theta_2$ using \eqref{DE}, we obtain the proposition. 
\end{proof}

To compute the period along $Z_2$ (see \eqref{period-eq4}), we prepare the following. 

\begin{prop}\label{prop det}
Assume $m>l$, $m\equiv k\pmod{l}$, 
and that $\mu:=m/l$ satisfies $\mu\not\equiv 0$, $\a^\chi$, $\b^\chi$, $\a^\chi+\b^\chi\pmod\Z$. 
Let $\Theta$ be as in Proposition \ref{Theta}. Then we have
\begin{align*}
&
\begin{pmatrix}
\int_{Z_1}s^{m-1}ds~\omega^\chi&\int_{Z_2}s^{m-1}ds~\omega^\chi\\
\int_{Z_1}s^{m-l-1}ds~\omega^\chi&\int_{Z_2}s^{m-l-1}ds~\omega^\chi
\end{pmatrix}
\\& =2\pi i(1-\zeta^m)
\begin{pmatrix}1&0\\0& \frac{1}{\m-1}\end{pmatrix}
\begin{pmatrix}
\Theta F_\m(\l) &\Theta G_\m(\l) \\ \pl \Theta F_\m(\l) & \pl \Theta G_\m(\l)
\end{pmatrix}
\begin{pmatrix}1& \xi\\0& 1-\xi\end{pmatrix}.
\end{align*}
Here $\int_{Z_i}s^{m-l-1}ds~\omega^\chi$ is also regarded 
as an analytic function of variable $\l$ as in Proposition \ref{prop P_m}.
\end{prop}

\begin{proof}
Firstly, since $\sigma_\zeta$ acts on $s^{m-1}ds$ as multiplication by $\zeta^m$, we have, by Proposition \ref{Theta}, 
$$\int_{Z_1}s^{m-1}ds~\omega^\chi=(1-\zeta^m)P_m(\l)=2\pi i (1-\zeta^m)\Theta F_\m(\l).$$
Secondly, $G_\m(\l)$ is a solution of the differential equation \eqref{DE}, 
and its monodromy at $\l=0$ is given by
\begin{equation*}\label{monodromy F}
T_{\l=0}(F_\m,G_\m)=(F_\m,G_\m)
\begin{pmatrix} \xi & 0\\1-\xi & 1\end{pmatrix}, \quad 
\xi:=e^{2\pi i (\m-\a^\chi-\b^\chi)}
\end{equation*}
(see \cite[\S2.9 (43)]{bateman}). 
Note that $\xi$ depends only on $\mu$ mod $\Z$ and $\xi \ne 1$ by the assumption. 
Hence we have
$$\int_{Z_2}s^{m-1}ds~\omega^\chi
=2\pi i (1-\zeta^m)\Theta T_{\l=0}F_\m(\l)
=2\pi i (1-\zeta^m)\Theta(\xi F_\m(\l)+(1-\xi)G_\m(\l)).$$
Then the computation for $\int_{Z_i}s^{m-l-1}ds~\omega^\chi$ ($i=1,2$) follows by Proposition \ref{prop P_m}. 
\end{proof}

Now we finish the proof of Theorem \ref{main1}. 
By Proposition \ref{prop det}, it suffices to show
$$\begin{vmatrix}
\Theta F_\m(\l) &\Theta G_\m(\l) \\ \pl \Theta F_\m(\l) & \pl \Theta G_\m(\l)
\end{vmatrix} \ne 0$$
for some $m$. 
This is equivalent to that $\Theta F_\m(\l)/\Theta G_\m(\l)$ is non-constant. 
Suppose that $\Theta F_\m(\l)=C\Theta G_\m(\l)$ for some constant $C$. 
Then $F_\m(\l)-CG_\m(\l)$ is a solution of both $\Theta f=0$ and \eqref{DE}. 
Since \eqref{DE} is irreducible by the assumption that $\m\not\equiv\a$, $\b$, $\a+\b \pmod\Z$, and the order of $\Theta$ is one, 
we have $\Theta=0$. 
Suppose that this is the case for any $m>l$ with $m \equiv k \pmod{l}$. Then  
$P_m(\l)=\int_\Delta s^{m-1}ds\,\omega^\chi=0$
for any such $m$. 
Applying the elementary lemma below (replace $s$ with $x^{1/l}$), 
it follows that $\int_\delta \omega^\chi=0$, hence a contradiction. 
\qed

\begin{lem}Let $f$ be a continuous function on the closed interval $[0,1]$ whose zeros have no accumulation point. 
If $\int_0^1 f(x) x^{n-1}dx=0$ for all $n\in \Z_{> 0}$, then $f\equiv 0$. 
\end{lem}

\begin{proof}
By replacing $f(x)$ with $(x-1)f(x)$ if necessary, we can assume $f(1)=0$. 
Suppose that $f \not\equiv 0$. By replacing $f$ with $-f$ if necessary, there exists $a \in (0,1)$ such that $f(x)>0$ for any $x \in (a,1)$. 
Put $M=\max_{x\in [0,a]}|f(x)|$ and choose $b, c \in (a,1)$ ($b<c$) and $m>0$ such that 
$f(x)\ge m$ for any $x \in [b,c]$. 
Then
\begin{align*}
\int_0^1 f(x)x^{n-1}\,dx 
&> -M \int_0^a x^{n-1}\, dx+m\int_b^c x^{n-1}\, dx
\\&=\frac{mc^n}{n}\left(1-\left(\frac{b}{c}\right)^n-\frac{M}{m}\left(\frac{a}{c}\right)^n\right).
\end{align*}
Since the right-hand side is positive for sufficiently large $n$, this contradicts the assumption. 
\end{proof}

\section{Regulator Formula}\label{reg-sect}
We keep the setting and the notations in \S \ref{period-sect1}.
Put
\begin{equation}\label{reg-eq02}
C:=\Gr^W_2\psi_{t=1}\cM\cong \pi_*\Q|_{\{1\}\times S}\ot (\Gr^W_2\psi_{t=1}R^1f_*\Q)
\end{equation}
a VHdR on $S$.
In this section we discuss the exact sequences
\begin{equation}\label{reg-eq00}
\xymatrix{
0\ar[r]& W_2\cH(e)\ar[r]&\cH(e)\ar[r]&(\cH/W_2\cH)(e)\ar[r]&0\\
0\ar[r]& W_2\cH(e)\ar[r]\ar@{=}[u]&\cH^\prime(e)\ar[r]\ar[u]
&C(e)\ot\Q(-1)\ar[r]\ar@{^{(}->}[u]&0
}
\end{equation}
of mixed Hodge-de Rham structures on $S=\P^1\setminus\{0,1,\infty\}$
arising from \eqref{period-eq1}, where the right vertical inclusion is as in
Proposition \ref{HG-prop3}.
Since 
$\Gr^W_2\psi_{t=1}(R^1f_*\Q)(e_0)\cong E_0$
is a constant VHdR of type $(1,1)$, $C(e)$ is one-dimensional over $E$ and endowed with Hodge type $(1,1)$ (however the monodromy is nontrivial). 

\subsection{Setting}
Let $Q:R^1f_*\Q\ot R^1f_*\Q\to \Q(-1)$ be a polarization form which also induces
a polarization on the $e_0$-part $(R^1f_*\Q)(e_0)$.
It naturally extends to a non-degenerate pairing $Q:\cM\ot \cM\to \Q(-1)$
which is compatible with the action of $\Aut(\pi)\cong \mu_l$, namely
$Q(\sigma x,\sigma y)=Q(x,y)$ for $\sigma\in\Aut(\pi)$.
This also  induces a polarization on the $e$-part $\cM(e)$.
We have isomorphisms
\begin{equation}\label{reg-eq0}
(\cM)^*\cong \cM\ot\Q(1),\quad
(\cM(e))^*\cong \cM(e)\ot\Q(1)
\end{equation}
induced from $Q$ where $(-)^*$ denotes the dual sheaf.
Let $j:\cU\hra  \P^1_S$ and $\pr_2:\P^1_S=\P^1\times S\to S$.
Then there are isomorphisms
\begin{equation}\label{reg-eq1}
(\cH)^*=(R^1\pr_{2*}Rj_*\cM)^*\cong 
R^1\pr_{2*}j_!\cM^*\ot\Q(1)\cong 
R^1\pr_{2*}j_!\cM\ot\Q(2)
\end{equation}
induced from the Verdier duality and \eqref{reg-eq0}.
We show that \eqref{reg-eq1} induces an isomorphism
\begin{equation}\label{reg-eq2}
(W_2\cH)^*\cong W_2\cH\ot\Q(2).
\end{equation}
Let $i:Z=\P^1_S\setminus\cU\hra\P^1_S$ be the complement.
Note that $Z$ is finite etale over $S$.
There is an exact sequence 
\[
0\lra i_*i^*j_*\cM\lra j_!\cM[1]\lra j_*\cM[1]\lra0
\]
of mixed Hodge modules where $j_*\cM=R^0j_*\cM$.
Applying $R\pr_{2*}$, one has
\[
\xymatrix{
0\ar[r]&i_*i^*j_*\cM\ar[r]&
R^1\pr_{2*}j_!\cM\ar[r]& R^1\pr_{2*}j_*\cM
\ar[r]\ar@{=}[d]^{\eqref{period-eq2-1}}& 0\\
&&&W_2\cH
}
\]
because 
$(R^0\pr_{2*}j_*\cM)_\lam=H^0(U_\lam,\cM_\lam)=0$ for $\lam\in \C\setminus\{0,1\}$
(see \S \ref{period-sect3} for the notation).
The mixed Hodge-de Rham structure
\[
i_*i^*j_*\cM_\lam=\bigoplus_{p=0,1,\lam,\infty}\ker[T_p-1:\psi_{t=p}\cM_\lam
\to\psi_{t=p}\cM_\lam\ot\Q(-1)]
\]
has weight $\leq 1$.
This implies $\Gr^W_2R^1\pr_{2*}j_!\cM
=R^1\pr_{2*}j_*\cM=W_2\cH$.
We thus have \eqref{reg-eq2} by taking the graded piece of \eqref{reg-eq1} of weight $-2$.

The isomorphisms 
\eqref{reg-eq0} and \eqref{reg-eq2} are {\it not} compatible with respect to
the multiplication by $R$.
Here the multiplication on the left hand side of \eqref{reg-eq0} or \eqref{reg-eq2} is given 
as in \S \ref{HdR-sect3}. 
For $r\in E$, we denote by ${}^tr$ the multiplication on $\cM(e)$ such that
\[
Q(rx,y)=Q(x,{}^try),\quad \forall~x,y.
\]
The multiplication by $r$ on the left corresponds to 
${}^tr$ in the right of \eqref{reg-eq0}.
Note ${}^t\sigma=\sigma^{-1}$ for $\sigma\in\Aut(\pi)$.
For $\chi:R\to\ol\Q$, we denote $(-)({}^t\chi)$
the subspace on which ${}^tr$ acts by
multiplication by $\chi(r)$ for all $r\in E$. 
Then \eqref{reg-eq0} and \eqref{reg-eq2} induce
\[
(\cM(\chi))^*\cong \cM({}^t\chi),\quad
(W_2\cH(\chi))^*\cong W_2\cH({}^t\chi).
\]
\subsection{Theorem on regulators}
Let $\O^\zar$ be 
the Zariski sheaf of polynomial functions (with coefficients in $\ol\Q$) on 
$S=\bA^1_{\ol\Q}\setminus\{0,1\}$ with coordinate $\lambda$.
Let $\O^{an}$ be the sheaf of analytic functions on $S^{an}=\C^{an}\setminus\{0,1\}$.
Let $a$
be the canonical morphism from the analytic site to the Zariski site.
Set
\begin{align*}
&\cJ:=\Coker[a^{-1}F^2W_2\cH_\dR\op\iota^{-1}W_2\cH_B\to \O^{an}\ot_{a^{-1}\O^\zar}a^{-1}W_2\cH_\dR],
\\&
\cJ^*:=
\Coker[a^{-1}{\mathcal Hom}(W_2\cH_\dR/F^1,\O^\zar)\op\iota^{-1}W_2\cH_B^*\to
{\mathcal Hom}(a^{-1}W_2\cH_\dR,\O^{an})]
\end{align*}
sheaves on the analytic site $\C^{an}\setminus\{0,1\}$.
Note $\cJ^*\cong \cJ$ by \eqref{reg-eq2}.

Let $C$ be the VHdR on $S$ as defined in \eqref{reg-eq02}.
Let $h:\wt{S}\to S$ be a generically finite and dominant map
such that $\sqrt[l]{\lambda-1}\in \ol\Q(\wt{S})$. 
Then $h^*C$ is a constant VHdR of type $(1,1)$.
Let
\[
\delta:h^*C(e)\ot\Q(1)=\Hom_{\mathrm{VMHdR}}(\Q,h^*C\ot\Q(1))\lra 
\Ext^1_{\mathrm{VMHdR}}(\Q,h^*W_2\cH(e)\ot\Q(2)).
\]
be the connecting homomorphism
arising from the exact sequence of \eqref{reg-eq00}.
Let $\rho$ be the composition of maps
\begin{align*}
h^*C(e)\ot\Q(1)&\os{\delta}{\lra} 
\Ext^1_{\mathrm{VMHdR}}(\Q,h^*W_2\cH\ot\Q(2))\\
&\lra \vg(\wt{S}^{an},h^*\cJ(e))\\
&\os{\sim}{\lra} \vg(\wt{S}^{an},h^*\cJ^*(e))
\end{align*}
where the second arrow is constructed 
in a similar way to the proof of Proposition \ref{HdR-prop1}.
Let $\rho({}^t\chi)$ be ${}^t\chi$-part of $\rho$, namely 
the composition of maps 
\begin{align*}
\ol\Q\cong (h^*C(e)\ot\Q(1))({}^t\chi)
&\lra\Ext^1_\FiltBdR(\ol\Q,h^*W_2\cH({}^t\chi)\ot\Q(2))\\
&\os{\sim}{\lra} \Ext^1_{\FiltBdR}(\ol\Q,h^*(W_2\cH)^*(\chi))\\
&\lra\vg(\wt{S}^{an},h^*\cJ^*(\chi)).
\end{align*}
Here $\cJ^*(\chi)$ is defined by replacing $W_2\cH_\dR$ with $W_2\cH_\dR(\chi)$,
and $W_2\cH_B^*$ with $W_2\cH_B^*(\chi)
={\mathcal Hom}_{\ol\Q}(W_2\cH_B(\chi),\ol\Q)$. 

\begin{thm}[Regulator formula]\label{main2}
Let the assumptions, $\m$ and $\Theta$ be as in Theorem \ref{main1}. 
Let
$\rho({}^t\chi)(1)\in 
(h^*\O^{an})^2\cong {\mathcal Hom}(a^{-1}W_2\cH_\dR(\chi),h^*\O^{an})$
be a local lifting where the isomorphism is with respect to a $\ol\Q$-frame of
$W_2\cH_\dR(\chi)$. Then we have
$$\rho({}^t\chi)(1) \equiv (\Theta H_\m(\l), (\m-1)^{-1}\pl\Theta H_\m(\l)) \pmod{\ol{\Q(\l)}^2}, $$
where we put
$$H_\m(\l):=\frac{1}{(1-\a^\chi)(1-\b^\chi)}(\l-1)^{\m-1} \FFF{1,1,1-\m}{2-\a^\chi,2-\b^\chi}{\frac{1}{1-\l}}. $$
\end{thm}

The map $\rho$ is related to {\it Beilinson's regulator map} in the following way.
Let $\P^1_{\wt{S}}:=\P^1\times \wt{S}$ 
and $\pi:\P^1_{\wt{S}}\to \P^1$ given by $(s,\lambda)\mapsto
\lambda-h(s)^l$. Let
\[
\xymatrix{
X_{\wt{S}}\ar[r]^{i\qquad}\ar[rd]_{f_{\wt{S}}}&\P^1_{\wt{S}}
\times_{\P^1}X\ar[r]\ar[d]&X\ar[d]^f\\
&\P^1_{\wt{S}}\ar[r]^\pi&\P^1
}
\]
with $i$ desingularization.
Let $g:=\pr_2\circ f_{\wt{S}}:X_{\wt{S}}\to \wt{S}$ be a projective smooth map.
Let
\[
\reg:H^3_\cM(X_{\wt{S}},\Q(2))\lra \Ext^1_{\mathrm{VMHdR}}(\Q,R^2g_{*}\Q(2))
\]
be the Beilinson regulator map.
Letting $(R^2g_*\Q)_0:=\ker[R^2g_*\Q\to \pr_{2*}R^2f_{\wt{S}*}\Q]$, then
there is a canonical surjective map
$(R^2g_*\Q)_0\to h^*W_2\cH$ (cf. \eqref{period-sect5-eq2}), so that we have
\[
\reg_0:H^3_\cM(X_{\wt{S}},\Q(2))_0\lra \Ext^1_{\mathrm{VMHdR}}(\Q,h^*W_2\cH\ot\Q(2))
\]
where $H^3_\cM(X_{\wt{S}},\Q(2))_0\subset H^3_\cM(X_{\wt{S}},\Q(2))$ 
is the inverse image of
$\Ext^1_{\mathrm{VMHdR}}(\Q,(R^2g_{*}\Q(2))_0)$ by $\reg$.
Let $D_{\wt{S}}\subset X_{\wt{S}}$ 
be the inverse image of singular fibers $D\subset X$ of $f$.
Then there is a canonical map
\[
H^3_{\cM,D_{\wt{S}}}(X_{\wt{S}},\Q(2))\to H_{B,D_S}^3(X_{\wt{S}},\Q(2))\cap H^{0,0}\to
\Hom_{\mathrm{VMHdR}}(\Q,h^*\cH/W_2\ot\Q(2))
\]
from the motivic cohomology group with support in $D_{\wt{S}}$.
Then the following diagram is commutative
\[
\xymatrix{
H^3_{\cM,D_{\wt{S}}}(X_{\wt{S}},\Q(2))\ar[r]\ar[d]&
\Hom_{\mathrm{VMHdR}}(\Q,h^*\cH/W_2\ot\Q(2))\ar[d]\ar@/^30pt/[dd]<17ex>^\rho\\
H^3_\cM(X_{\wt{S}},\Q(2))_0\ar[r]^{\reg_0\quad\qquad }&
\Ext^1_{\mathrm{VMHdR}}(\Q,h^*W_2\cH\ot\Q(2))\ar[d]\\
&\vg(\wt{S}^{an},h^*\cJ^*)
}
\]
\subsection{Proof of Regulator formula : Part 1}\label{reg-sect3}
Let
\[
\xymatrix{
X_\lam\ar[r]^{i\qquad}\ar[rd]_{f_\lam}&\P^1_\lam
\times_{\P^1}X\ar[r]\ar[d]&X\ar[d]^f\\
&\P^1_\lam\ar[r]^{\pi_\lam}&\P^1
}
\]
be as in \S \ref{period-sect5}.
Let $D_\lam\subset X_\lam$ be the inverse image of the singular fibers of $f$.
We denote the coordinate of $\P^1$ (resp. $\P^1_\lam$) by $t$ (resp. $s$).
We also use the notation in \S \ref{period-sect3}.
Note $\cM_{\lam}\cong \pi_{\lam*}R^1f_{\lam*}\Q$
is endowed with de Rham structure induced from
a $\ol\Q$-frame on $\cM_\dR$.
The distinguished triangle
\[
 0\lra j_!\pi_{\lam*}f_{\lam*}\Q\lra  j_!\tau_{\leq 1}
 \pi_{\lam*}Rf_{\lam*}\Q\lra
 j_!\pi_{\lam*}R^1f_{\lam*}\Q[-1]\lra0
\]
and the fact that
\[
H^2(\P^1_\lam,j_!\pi_{\lam*}\Q)
=H^0(\P^1_\lam,Rj_*\pi_{\lam*}\Q)^*
=H^0(\P^1_\lam\setminus\{0,1,\lam,\infty\},\pi_{\lam*}\Q)^*=0
\]
implies
\[
H^2(\P^1,  j_!\pi_{\lam*}\tau_{\leq 1}Rf_{\lam*}\Q)
\os{\cong}{\lra}
H^1(\P^1, j_!\cM_{\lam}).
\]
Hence we have an injective map
\[
H^1(\P^1, j_!\cM_{\lam})\lra H^2(\P^1,j_!\pi_{\lam*}Rf_{\lam*}\Q)
=H^2(\P^1,\pi_{\lam*}Rf_{\lam*}j_!\Q)=H^2(X_\lam,D_\lam;\Q).
\]
We have a commutative diagram with exact rows
\[{\small
\xymatrix{
0\ar[r] &W_2H_\lam(2) \ar@{=}[d]\ar[r]&H_\lam(2)\ar[r]\ar@{=}[d]
&H_\lam/W_2(2)\ar[r]\ar@{=}[d]&0\\
0\ar[r] &(W_2H_\lam)^* \ar[r]&H^1(\P^1, j_!\cM_{\lam})^*\ar[r]&
H^1(\P^1, j_!\cM_{\lam})^*/W_{-2}\ar[r]&0\\
0\ar[r] &H_2(X_\lam,\Q)/H_2(D_\lam) \ar[u]^{a_1}\ar[r]
&H_2(X_\lam,D_\lam;\Q)
\ar[u]^{a_2}\ar[r]&
H_1(D_\lam)\ar[u]^{a_3}\ar[r]&0\\
0\ar[r] &H_2(\ol{U}_\lam,\Q)/H_2(D^\circ_\lam) \ar[u]^{b_1}\ar[r]
&H_2(\ol{U}_\lam,D_\lam^\circ;\Q)
\ar[u]^{b_2}\ar[r]&H_1(D_\lam^\circ)\ar[u]^{b_2}
}}
\]
where $D_\lam^\circ:=D_\lam\cap \ol{U}_\lam$.
Since $a_2$ is surjective, so are $a_1$ and $a_3$.
Moreover $a_3$ is bijective (local invariant cycle theorem).

Let $Z_1,Z_2$ be the homology cycles \eqref{period-eq4}, and
$(m_1,m_2)=(l\m, l\m-l)$ where $\m$ is as in Theorem \ref{main1}. 
Note that $(W_2H_\lam)^*(\chi)$ is spanned by the images of $Z_1$ and $Z_2$.
Hence
\begin{align*}
&H_{\lam,B}({}^t\chi)/W_2\cap H^{0,0}\cong\ol\Q
\\&\os{=}{\lra} 
H_1^B(D^\circ_\lam)(\chi)\cap H^{0,0}\\
&\lra 
\Ext^1_\FiltBdR(\ol\Q,H_2(\ol{U}_\lam,\Q)/H_2(D^\circ_\lam)(\chi))\\
&\lra
\Ext^1_\FiltBdR(\ol\Q,(W_2H_\lam)^*(\chi))\\
&\os{\sim}{\lra}\Coker[(W_2H_{\lam,B})^*(\chi)\op\Hom(W_2H_{\lam,\dR}/F^1,\ol\Q)\to
\Hom(W_2H_{\lam,\dR}(\chi),\C)]\\
&\os{=}{\lra}\Coker[\Hom(W_2H_{\lam,\dR}(\chi)/F^1,\ol\Q)\to
\Hom(W_2H_{\lam,\dR}(\chi),\C)/\Image H_2(\ol{U}_\lam,\ol\Q)]
\end{align*}
and the composition of the above coincides with the restriction 
$\rho({}^t\chi)|_{\l=\lam}$.
Moreover the image of $H_2(\ol{U}_\lam,\ol\Q)$ is given by the period matrix
\[
\Per(W_2\cH(\chi))|_{\l=\lam}=\begin{pmatrix}
\int_{Z_1}s^{m_1-1}ds\omega&\int_{Z_2}s^{m_1-1}ds\omega\\
\int_{Z_1}s^{m_2-1}ds\omega&\int_{Z_2}s^{m_2-1}ds\omega
\end{pmatrix}\bigg|_{\l=\lam}:\ol\Q^2\lra\C^2
\]
under the isomorphism
$$\Hom(W_2H_{\lam,\dR}(\chi),\C)\cong \C^2$$
given by the $\ol\Q$-basis 
$\{s^{m_1-1}ds\omega,s^{m_2-1}ds\omega\}$ of $W_2H_{\lam,\dR}(\chi)$.
Let $D_\lam^{ss}\subset D_\lam$ be the inverse image of $f^{-1}(1)$.
Note $D^{ss}_\lam\subset \ol{U}_\lam$.
Then for $1\in \ol\Q\cong H_1(D^{ss}_\lam)(\chi)\subset
H_1(D^\circ_\lam)(\chi)\cap H^{0,0}$, we want to compute a lifting
\[
\rho({}^t\chi)(1)=(\phi_1(\lam),\phi_2(\lam))\in\C^2.
\]
\begin{lem}\label{reg-lem1}
Let $\Gamma(\lam)\in H_2^B(\ol{U}_\lam,(\pi_\lam f_\lam)^{-1}(1);\ol\Q)$
be a lifting of the homology cycle $1\in \ol\Q\cong H_1(D^{ss}_\lam)(\chi)$.
Then there is an algebraic function $R(\lambda)\in \ol{\Q(\lambda)}$ 
of variable $\lambda$ such that
\[
\phi_i(\lam)=\int_{\Gamma(\lam)}s^{m_i-1}ds~\omega^\chi|_{\l=\lam}+R(\lam)
\]
for $\lam$ in a small neighbourhood of $\C^{an}\setminus\{0,1\}$.
\end{lem}
\begin{proof}
See \cite{a-o}, Proposition 7.2; the situation there is slightly different but the same discussion works.
\end{proof}
\subsection{Proof of Regulator formula : Part 2}\label{reg-sect4}
Recall from \S \ref{period-sect4} the homology cycles
$\delta,\gamma\in H_1(X_t,\ol\Q)(\chi)$.
By the local invariant cycle theorem, there is an exact sequence
\begin{equation}\label{reg-eq3}
H_1(X_t,\Q)\os{T_0-1}{\lra} H_1(X_t,\Q)\lra  H_1(f^{-1}(0),\Q)\lra 0, 
\end{equation}
where $T_0$ is the local monodromy at $t=0$.
We note that $H_1(f^{-1}(0),\Q)$ has multiplication by $R_0$ induced 
from \eqref{reg-eq3} (recall from \S \ref{period-sect1} that 
$R^1f_*\Q$ has multiplication by $R_0$).
An element $\gamma' \in H_1(X_t,\ol\Q)(\chi)$ vanishes as $t\to0$ 
if and only if it belongs to the
one-dimensional space
\[
\ker[H_1(X_t,\ol\Q)(\chi)\to H_1(f^{-1}(0),\ol\Q)]=
\Image[T_0-1:H_1(X_t,\ol\Q)(\chi)\to H_1(X_t,\ol\Q)(\chi)].
\]

\begin{lem}\label{gamma'}
Put
$$\gamma':=\gamma+\frac{1-e^{2\pi i(\a^\chi+\b^\chi)}}{(1-e^{2\pi i\a^\chi})(1-e^{2\pi i\b^\chi})}\delta.$$
Then $\gamma'$ is a basis of 
$\ker[H_1(X_t,\ol\Q)(\chi)\to H_1(f^{-1}(0),\ol\Q)]$, 
and we have
$$\int_{\gamma'} \omega^\chi=B(1-\a^\chi,1-\b^\chi)\theta \left(t^{1-\a^\chi-\b^\chi}\FF{1-\a^\chi,1-\b^\chi}{2-\a^\chi-\b^\chi}{t}\right).$$
\end{lem}

\begin{proof}
If $\a^\chi+\b^\chi=1$, then $T_0$ is unipotent and $\Image(T_0-1)=\Ker(T_0-1)$, to which 
$\gamma'=\gamma$ belongs. The assertion about the period is Lemma \ref{key-lem1}. 
Suppose $\a^\chi+\b^\chi\ne 1$. 
Recall the notations \eqref{f_1 f_2}, \eqref{f_3} and the relation \eqref{f123}. 
The assertion about the period follows from \eqref{f123}. 
Since $f_1(t)$ and $f_3(t)$ form a basis of the local solutions near $t=0$, 
and $T_0-1$ annihilates $f_1(t)$ but not $f_3(t)$, the $\gamma'$ generates $\Image(T_0-1)$. 
\end{proof}

We regard $\gamma'$ as a homology cycle in a general fiber of $f_\lam$.
Fix $l$th roots $\sqrt[l]{\lam}$ and $\sqrt[l]{\lam-1}$.
Let $\Gamma_\lam\subset \ol{U}_\lam$ be the Lefschetz thimble over a path 
from $s=\sqrt[l]{\lam-1}$ to $s=\sqrt[l]{\lam}$ with fiber $\gamma'$
($s$ is the coordinate of $\P^1_\lam$).
Note that $\delta$ vanishes as $t\to1$ but $\gamma$ does not. 
Therefore $\Gamma_\lam$ has a nontrivial boundary supported on 
$f_\lam^{-1}(\sqrt[l]{\lam-1})\cong f^{-1}(1)$:
\[
\Gamma_\lam\in H^B_2(\ol{U}_\lam,f_\lam^{-1}(\sqrt[l]{\lam-1});\ol\Q),
\quad \partial\Gamma_\lam\ne0
\in H^B_1(f_\lam^{-1}(\sqrt[l]{\lam-1}))\cong H^B_1(f^{-1}(1)).
\]
Note that $H_1^B(f^{-1}(1),\Q)$ has multiplication by $R_0$ via the local invariant cycle theorem (cf. \eqref{reg-eq3}).
The $\chi|_{R_0}$-part $H^B_1(f^{-1}(1),\ol\Q)(\chi|_{R_0})$
is one-dimensional, spanned by $\partial\Gamma_\lam$.
Hence the $\chi$-part
\[
H_1^B((\pi_\lam f_\lam)^{-1}(1),\ol\Q)(\chi)\cong
H^B_1(f^{-1}(1),\ol\Q)(\chi|_{R_0})\ot H_0^B(\pi_\lam^{-1}(1))(\chi|_{\mu_l})
\]
is one-dimensional, spanned by the sum 
$\sum_{\sigma\in\mu_l} \chi(\sigma)^{-1}\cdot\sigma(\partial\Gamma_\lam)$.
Therefore, in Lemma \ref{reg-lem1} we may take $\Gamma(\lam)$ to be 
the sum $\sum_{\sigma\in\mu_l} \chi(\sigma)^{-1}\cdot\sigma\Gamma_\lam$.
Then we have
$$\int_{\Gamma(\lam)}s^{m-1}ds~\omega^\chi|_{\l=\lam}
=\sum_{\sigma\in\mu_l} \chi(\sigma)^{-1}
\int_{\sigma\Gamma_\lam}s^{m-1}ds~\omega^\chi|_{\l=\lam}
=l\int_{\Gamma_\lam} s^{m-1}ds~\omega^\chi|_{\l=\lam}.$$

\begin{lem}\label{lemma H}
Let $H_\mu(\l)$ be as defined in Theorem \ref{main2}. 
Then we have
\begin{align}\label{H}
H_\m(\l)
=B(1-\a^\chi,1-\b^\chi)
\int_0^1(\l-t)^{\m-1}t^{1-\a^\chi-\b^\chi}\FF{1-\a^\chi,1-\b^\chi}{2-\a^\chi-\b^\chi}{t}\, dt, 
\end{align}
and 
\begin{equation}\label{derive H}
\pl H_\m(\l)=(\m-1) H_{\m-1}(\l).
\end{equation}
\end{lem}

\begin{proof}
Recall the integral representation of ${}_2F_1$ (cf \cite{slater}, (4.1.2))
\begin{equation}\label{int rep 2F1}
B(b,c-b)\FF{a,b}{c}{x}=\int_0^1(1-xt)^{-a}t^{b-1}(1-t)^{c-b-1}\, dt.
\end{equation}
Applying this, we have, writing $\a=\a^\chi$, $\b=\b^\chi$, 
$$B(1-\a,1-\b)\FF{1-\a,1-\b}{2-\a-\b}{t}=\int_0^1(1-ts)^{\a-1}s^{-\b}(1-s)^{-\a}\,ds. $$
Letting $u=1-t$, $v=1-st$, we have
\begin{align*}
&\int_0^1\int_0^1(\l-t)^{\m-1}t^{1-\a-\b}(1-ts)^{\a-1}s^{-\b}(1-s)^{-\a}\,ds\, dt
\\&=(\l-1)^{\m-1} \int_0^1 \int_0^v \left(1-\frac{u}{1-\l}\right)^{\m-1}v^{\a-1}(v-u)^{-\a}(1-v)^{-\b}\, du\, dv
\\&=(\l-1)^{\m-1}\int_0^1\int_0^1 \left(1-\frac{wv}{1-\l}\right)^{\m-1}(1-w)^{-\a}(1-v)^{-\b}\, dw\, dv.
\end{align*}
Then, using \eqref{int rep 2F1} and \eqref{int rep 3F2}, we obtain \eqref{H}. 
Since 
$$\pl \FFF{1,1,1-\m}{2-\a,2-\b}{\frac{1}{1-\l}}=\frac{1-\m}{(2-\a)(2-\b)} \FFF{2,2,2-\m}{3-\a,3-\b}{\frac{1}{1-\l}}$$
similarly as \eqref{partial 2F1}, the proof of \eqref{derive H} amounts to show 
\begin{align*} 
& \frac{(1)_n(1)_n(1-\m)_n}{(2-\a)_n(2-\b)_n(1)_n}+ \frac{1}{(2-\a)(2-\b)}
\cdot \frac{(2)_{n-1}(2)_{n-1}(2-\m)_{n-1}}{(3-\a)_{n-1}(3-\b)_{n-1}(1)_{n-1}}
\\&=\frac{(1)_n(1)_n(2-\m)_n}{(2-\a)_n(2-\b)_n(1)_n} \quad (n>0), 
\end{align*}
and this is elementary. 
\end{proof}

\begin{prop}
For a positive integer $m \equiv k \pmod l$, put
$$Q_m(\lam):=\int_{\Gamma_\lam}s^{m-1}ds~\omega^\chi|_{\l=\lam}, \quad \m=\frac{m}{l}$$
regarded as an analytic function for $\lam$.
Then we have
\begin{equation}\label{Q_m}
Q_m(\l)= \frac{1}{l}\sum_{i=0}^N (a_i(\l)+b_i(\l)\pl) H_{\m+i}(\l), 
\end{equation}
where $a_i(\l)$, $b_i(\l)$ are as defined in \eqref{p to ab}. 
Moreover we have, if $\m>1$, 
\begin{equation}\label{partial Q_m}
\pl Q_m(\l)=(\m-1)Q_{m-l}(\l). 
\end{equation}
\end{prop}

\begin{proof}
Since $\frac{-dt}{\l-t}=l\frac{ds}{s}$, we have by Lemma \ref{gamma'}
$$Q_m(\l)=\frac{1}{l}B(1-\a^\chi,1-\b^\chi)
\int_0^1(\l-t)^{\m-1}\theta\left(t^{1-\a^\chi-\b^\chi}\FF{1-\a^\chi,1-\b^\chi}{2-\a^\chi-\b^\chi}{t}\right)dt. $$
By Lemma \ref{lemma H}, the same argument as in the proof of Proposition \ref{prop P_m} works to prove the proposition. 
\end{proof}

\begin{lem}\label{DH}
Let the differential operator $\cD$ be as defined in \eqref{DE}. Then we have
$$\cD H_\m(\l) =-(\l-1)^{\m-1}.$$
\end{lem}

\begin{proof}
Put $x=\frac{1}{1-\l}$, $F(x)=\FFF{1,1,1-\m}{2-\a^\chi,2-\b^\chi}{x}$ and $D=x\frac{d}{dx}$. Using $Dx^n=nx^n$, one easily verifies 
$$(x(D+1)(D+1-\m)-(D+1-\a^\chi)(D+1-\b^\chi))F=-(1-\a^\chi)(1-\b^\chi).$$
So $H_\m$ satisfies $\cD_1H_\m=-1$ with
\begin{align*}
\cD_1&=(x(D+1)(D+1-\m)-(D+1-\a^\chi)(D+1-\b^\chi))(-x)^{\m-1}
\\& =(-x)^{\m-1}(x(D+\m)D-(D-\a^\chi+\m)(D-\b^\chi+\m)). 
\end{align*}
Since $D=(1-\l)\partial_\l$, where $\partial_\l=\frac{d}{d\l}$, one obtains 
$\cD_1=(\l-1)^{1-\m}\cD$. Hence the lemma follows. 
\end{proof}

Now we finish the proof of Theorem \ref{main2}. 
Let $\mu=m/l$ be as in Theorem \ref{main1}.
By Lemma \ref{reg-lem1}, \eqref{Q_m} and \eqref{partial Q_m}, we have
\begin{align}
\phi_1(\lambda)&\equiv Q_m(\l)= \frac{1}{l}\sum_{i=0}^N (a_i(\l)+b_i(\l)\pl) H_{\m+i}(\l), 
\mod \ol{\Q(\lambda)},\label{reg-sect5-eq5}\\
\phi_2(\lambda)&\equiv Q_{m-l}(\l)=(\mu-1)^{-1}\partial_\l Q_m(\l),
\mod \ol{\Q(\lambda)}.\label{reg-sect5-eq6}
\end{align} 
By \eqref{derive H}, we have similarly as in the proof of Proposition \ref{Theta}, that 
$Q_m(\l)= \Theta_1 H_{\m+N}(\l)$ where $\Theta_1$
is the same differential operator as in the proof of Proposition \ref{Theta}.
By Lemma \ref{DH}, 
we have 
$$H_{\m+N}(\l)\equiv \Theta_2H_\m(\l) \mod{\ol{\Q(\l)}}$$
where $\Theta_2$ is the same differential operator in the proof of Proposition \ref{Theta}.
Hence $\phi_1(\l)\equiv \Theta_1\Theta_2 H_\m(\l)=\Theta H_\m(\l)$ as desired.
\qed

\subsection{Question of Golyshev}
We give an affirmative answer to the question of Golyshev in a special case.
\begin{lem}\label{reg-comp-lem1}
Let 
\[
P_{\mathrm{HG}}:=D_\lambda(D_\lambda-\mu+\alpha^\chi+\beta^\chi-1)
-\lambda(D_\lambda+\alpha^\chi-\mu)(D_\lambda+\beta^\chi-\mu)
\]
be the hypergeometric differential operator. Put 
\[
Q_{\mathrm{HG}}:=\theta_\lambda P_{\mathrm{HG}},\quad
\theta_\lambda:=(1-\lambda)D_\lambda+(\mu-1)\lambda, 
\]
and local systems of $\C$-modules on $S:=\P^1\setminus\{0,1,\infty\}$
\[
V_P:=\mathit{Sol}(D_S/D_SP_{\mathrm{HG}}),\quad
V_Q:=\mathit{Sol}(D_S/D_SQ_{\mathrm{HG}}),
\]
where 
$D_S$ denotes the ring of differential operators on $S$. 
Let 
\[
0\lra V_P\lra V_Q\lra V_Q/V_P\lra0
\]
be the exact sequence obtained by applying the solution functor 
$\mathit{Sol}(\bullet):=\mathcal{H}om_{D_S}
(\bullet,\O_S)$ on
\[
0\lra D_S/D_S\theta_\lambda
\lra D_S/D_SQ_{\mathrm{HG}}\lra D_S/D_SP_{\mathrm{HG}}\lra0.
\]
Then, for any generically finite dominant map
$h:T\to S$, the exact sequence
\begin{equation}\label{reg-comp-lem1-eq1}
0\lra h^*V_P\lra h^*V_Q\lra h^*(V_Q/V_P)\lra0
\end{equation}
of $\C[\pi_1(T)]$-modules does not split.
\end{lem}
\begin{proof}
We first note that $P_{\mathrm{HG}}=\lambda\cD$ where $\cD$ is the differential operator
\eqref{DE}.
Let $F_\mu(\lambda)$, $\cg(\lambda)$ be as in Theorem \ref{main1}, and
$H_\mu(\lambda)$ as in in Theorem \ref{main2}.
Then the solutions of $P_{\mathrm{HG}}$ are
$F_\mu(\lambda), \cg(\lambda)$ (cf. the proof of Propositions \ref{Theta} 
and \ref{prop det}), and
the solutions of $Q_{\mathrm{HG}}$ are
$F_\mu(\lambda),\cg(\lambda), H_\mu(\lambda)$ (cf. Lemma \ref{DH}):
\[
V_P=\langle F_\mu(\lambda),\, \cg(\lambda)\rangle_\C,\quad
V_Q=\langle F_\mu(\lambda),\, \cg(\lambda),\,H_\mu(\lambda)\rangle_\C.
\]
Since 
$\Ext_{\pi_1(S)}(V_Q/V_P,V_P)\to \Ext_{\pi_1(T)}(h^*(V_Q/V_P),h^*V_P)$ is injective, 
we may assume $T=S$.
Assume that the sequence \eqref{reg-comp-lem1-eq1} splits.
This means that there are $c_1,c_2\in\C$ such that
$\C(H_\mu(\lambda)+c_1F_\mu(\lambda) +c_2\cg(\lambda))$ is stable
under the action of $\pi_1(S,\lambda)$.
The eigenvalues of the local monodromy $T_\infty$ at $\lambda=\infty$
on $V_P$ are $e^{2\pi i(\alpha^\chi-\mu)}$, $e^{2\pi i(\beta^\chi-\mu)}$.
On the other hand 
$H_\mu(\lambda)$ is the eigenvector with eigenvalue $e^{-2\pi i\mu}$.
Therefore $c_1=c_2=0$, namely $H_\mu(\lambda)$ is stable
under the action of $\pi_1(S,\lambda)$.
The eigenvalues of the local monodromy $T_0$ at $\lambda=0$ on $V_Q$
(resp. $V_P$)
are $1$, $1$, $e^{2\pi i(\mu-\alpha^\chi-\beta^\chi)}$
(resp. $1$, $e^{2\pi i(\mu-\alpha^\chi-\beta^\chi)}$).
Therefore the eigenvalue of $T_0$ on $V_Q/V_P\cong \C$ is $1$, namely
the trivial action.
Therefore $T_1=T_\infty^{-1}$ acts on $H_\mu(\lambda)$ by multiplication by
$e^{2\pi i\mu}$.
Thus the function 
\[
(1-\alpha^\chi)(1-\beta^\chi)(\lambda-1)^{1-\mu}
H_\mu(\lambda)=
{}_3F_2\left({1,1,1-\mu\atop 2-\alpha^\chi,2-\beta^\chi};(1-\lambda)^{-1}\right)
\]
has the trivial monodromy, and this means that this is a rational function.
This is impossible. Indeed let $\sum_na_nz^n$ be the Laurent expansion of the above
with respect to variable $z=1-\lambda$. Then this satisfies a differential equation
\[
Q=(D_z-1)(D_z-1)(D_z-1+\mu)-zD_z(D_z-1+\alpha^\chi)(D_z-1+\beta^\chi).
\]
Hence
\[
(n-1)^2(n-1+\mu)a_n-(n-1)(n-2+\alpha^\chi)(n-2+\beta^\chi)a_{n-1}=0,\quad \forall\, n.
\]
Then $a_n=0$ for all $n\leq0$ as $a_n=0$ for $n\ll0$. Moreover
\[
a_n=\frac{(n-2+\alpha^\chi)(n-2+\beta^\chi)}{(n-1)(n-1+\mu)}a_{n-1}
=\frac{(\alpha^\chi)_{n-1}(\beta^\chi)_{n-1}}{(n-1)!(1+\mu)_{n-1}}a_1,\quad n\geq 1.
\]
We thus have
\[
\sum_na_nz^n=a_1z\cdot\,\FF{\alpha^\chi,\,\beta^\chi}{1+\mu}{z}.
\]
The right hand side has nontrivial monodromy unless $a_1=0$.
\end{proof}

\begin{thm}\label{reg-comp-cor2}
Let the notation and assumption be as in Theorem \ref{main2}.
Then the dual of the exact sequence 
\begin{equation}\label{reg-comp-cor2-eq1}
0\lra W_2\cH_\dR({}^t\chi)\lra
\cH^\prime_\dR({}^t\chi)\lra C({}^t\chi)\lra0
\end{equation}
of connections which underlies \eqref{reg-eq00} is isomorphic to  
\begin{equation}\label{reg-comp-cor2-eq2}
0\lra D_S/\theta_\lambda D_S
\lra D_S/D_SQ_{\mathrm{HG}}\lra D_S/D_SP_{\mathrm{HG}}\lra 0.
\end{equation}
In particular, the extension \eqref{reg-comp-cor2-eq1} is nontrivial
by Lemma \ref{reg-comp-lem1}.
In other words, the regulator $\rho({}^t\chi)$ in Theorem \ref{main2} does not vanish.
\end{thm}
\begin{proof}
By the Riemann-Hilbert correspondence, 
it is enough to show that there is an isomorphism
\begin{equation}\label{reg-comp-cor2-eq3}
\xymatrix{
0\ar[r]
&W_2\cH_B({}^t\chi)\ar[r]\ar[d]^\cong
&\cH^\prime_B({}^t\chi)\ar[r]\ar[d]^\cong
&C({}^t\chi)\ar[r]\ar[d]^\cong&0\\
0\ar[r]&V_P\ar[r]&V_P\ar[r]&V_Q/V_P\ar[r]&0
}
\end{equation}
where the top sequence is the underlying local systems of $\C$-modules.

We know that the local system $ W_2\cH_B({}^t\chi)
\cong (W_2\cH_B(\chi))^*$ are spanned
by $Z_1$, $Z_2$, and $C({}^t\chi)$ by the boundary of 
$\Gamma(\lambda)$ 
(see \eqref{period-eq4} and Lemma \ref{reg-lem1}
for the notation).
It is not hard to see that the monodromy representation of $C({}^t\chi)$
is isomorphic to that of $V_Q/V_P$ (cf. \eqref{reg-eq02}).
The monodromy representation of $\langle Z_1,Z_2\rangle_\C$
is isomorphic to that of
\begin{align*}
\left\langle\int_{Z_1}s^{m_1-1}ds\omega,\, \int_{Z_2}s^{m_1-1}ds\omega\right\rangle
&=\langle \Theta F_\mu(\lambda),\,\Theta\cg(\lambda)\rangle\\
&\os{\cong}{\leftarrow} \langle F_\mu(\lambda),\,\cg(\lambda)\rangle\\
&=V_P
\end{align*}
by Theorem \ref{main1} (Period formula).

We show that the monodromy representation of $\cH'_B({}^t\chi)
=\langle Z_1,Z_2,\Gamma(\lambda)\rangle$
is isomorphic to that of
\begin{equation}\label{reg-comp-cor2-eq4}
\int_{Z_1}s^{m_1-1}ds\omega,\quad \int_{Z_2}s^{m_1-1}ds\omega,
\quad \int_{\Gamma_x(\lambda)}s^{m_1-1}ds\omega.
\end{equation}
To do this we need to check that the above integrals are linearly independent
over $\C$.
The first and second integrals are spanned by
$\Theta F_\mu(\lambda)$, $\Theta\cg(1-\lambda)$ and they are linearly independent
by Theorem \ref{main1} (Period formula).
The 3rd integral is equal to $\Theta H_\mu(\lambda)$ modulo an algebraic function
by Theorem \ref{main2} (Regulator formula).
Suppose that the 3rd integral is a linear combination of the 1st and 2nd integrals.
Then
\[
\Theta H_\mu(\lambda)=c_1\Theta F_\mu(\lambda)+c_2\Theta\cg(\lambda)
+\mbox{(an algebraic function)},\quad (\exists c_i\in\C).
\]
Let $T_\infty$ be the local monodromy at $\lambda=\infty$.
Then the eigenvalues on $V_P=\langle\Theta F_\mu(\lambda),\Theta\cg(\lambda)
\rangle$ are $e^{2\pi i(\alpha^\chi-\mu)}$, $e^{2\pi i(\beta^\chi-\mu)}$
and $\Theta H_\mu(\lambda)$ is the eigenvector with eigenvalue $e^{-2\pi i\mu}$.
Applying 
$(T_\infty-e^{2\pi i(\alpha^\chi-\mu)})(T_\infty-e^{2\pi i(\beta^\chi-\mu)})$ to the above,
we have that $\Theta H_\mu(\lambda)$ is an algebraic function.
However since \eqref{reg-comp-lem1-eq1} is a nontrivial extension 
(Lemma \ref{reg-comp-lem1}),
there is some $g\in \C[\pi_1(S,\lambda)]$ such that 
$gH_\mu(\lambda)=c'_1F_\mu(\lambda)+c'_2\cg(\lambda)\ne0$.
Applying $\Theta$, we have that
$f(\lambda):=c'_1\Theta F_\mu(\lambda)+c'_2\Theta \cg(\lambda)=g\Theta
 H_\mu(\lambda)$
is also an algebraic function.
Since $\Theta F_\mu(\lambda)$, $\Theta\cg(\lambda)$ are linearly independent over $\C$,
$f(\lambda)\ne0$.
It generates the 2-dimensional space $\langle \Theta F_\mu(\lambda),\,\Theta\cg(\lambda)\rangle\cong V_P$
as $\C[\pi_1(S)]$-module since $V_P$ is irreducible.
Hence the monodromy representation of $V_P$ factors through a finite quotient.
This is a contradiction. Hence the three integrals \eqref{reg-comp-cor2-eq4}
are linearly independent over $\C$.

We now have that
the monodromy representation of $\cH'_B({}^t\chi)$ 
is isomorphic to that of
\[
\Theta F_\mu(\lambda),\,\Theta\cg(\lambda), \,
\Theta H_\mu(\lambda)+\mbox{(an algebraic function)}.
\]
Therefore letting $h:T\to S$ be a finite covering which trivializing the monodromy of
the algebraic function, we have an isomorphism
$h^*\cH^\prime_B({}^t\chi))\cong h^*V_Q$ in a canonical way.
Thus the extension data
$[h^*V_Q]\in \Ext^1_{\pi_1(T)}(h^*(V_Q/V_P),h^*V_P)$
coincides with 
$[h^*\cH^\prime_B({}^t\chi)]\in \Ext^1_{\pi_1(T)}(h^*\Coker(T_1-1)({}^t\chi),
h^*W_2\cH^\prime_B({}^t\chi))$ under the natural isomorphisms 
$V_P\cong W_2\cH^\prime_B({}^t\chi)$ and $V_Q/V_P\cong \Coker(T_1-1)({}^t\chi)$.
Now the assertion follows from the injectivity of 
$\Ext^1_{\pi_1(S)}(V_Q/V_P,V_P)\to
\Ext^1_{\pi_1(T)}(h^*(V_Q/V_P),h^*V_P)$.
This completes the proof. 
\end{proof}

\subsection{Complement : Precise formula of Regulators}
Applying the 3-term relation on ${}_3F_2$ (e.g. \cite{a-o} Lemma 7.5) to 
\eqref{reg-sect5-eq5} and \eqref{reg-sect5-eq6},
one can obtain a more explicit description of $\phi_i(\l)$
as $\ol\Q(\l)$-linear combinations of
$H_\mu(\l)$ and $H_{\mu-1}(\l)$:
\begin{equation}\label{reg-comp-eq1}
H_\mu(\lambda):=(1-\alpha^\chi)^{-1}(1-\beta^\chi)^{-1}(\lambda-1)^{\mu-1}
{}_3F_2\left({1,1,1-\mu\atop 2-\alpha^\chi,2-\beta^\chi};(1-\lambda)^{-1}\right),
\end{equation}
\begin{equation}\label{reg-comp-eq2}
H_{\mu-1}(\lambda):=(1-\alpha^\chi)^{-1}(1-\beta^\chi)^{-1}(\lambda-1)^{\mu-2}
{}_3F_2\left({1,1,2-\mu\atop 2-\alpha^\chi,2-\beta^\chi};(1-\lambda)^{-1}\right)
\end{equation}
where $\mu:=m/l$ (cf. Prop. \ref{prop det}).
The following theorem is used in \cite{a-o-log}.
\begin{thm}[Regulator formula -- precise version]\label{reg-comp-cor1}
Let the notation and assumption be as in Theorem \ref{main2}.
Let $\mu=m/l$ be as in Theorem \ref{main1}.
Let $a_i(\lambda),b_i(\lambda)$ be as in \eqref{p to ab}.
Put $a:=2-\alpha^\chi$, $b:=2-\beta^\chi$, and
\begin{align*}
e_i(s)&:=(-1)^i(a_i(\lambda)+(s+i)b_{i+1}(\lambda))(1-\lambda)^i\\
&=\begin{cases}
\left(\frac{d^ip_0(\lambda)}{d\lambda^i}+(s+i)\frac{d^{i+1}p_1(\lambda)}{d\lambda^{i+1}}\right)
\frac{(1-\lambda)^i}{i!}& i\geq0, \\
-(s-1)p_1(\lambda)/(1-\lambda)&i={-1},
\end{cases}\\
A(s)&:=\frac{s(a+b+2s-3-s(1-\lambda)^{-1})}{(a+s-1)(b+s-1)},\quad
B(s):=\frac{s(1-s)\lambda}{(a+s-1)(b+s-1)}
\end{align*}
with indeterminate $s$.
Define $C_i(s)$ and $D_i(s)$ by
\[
\begin{pmatrix}
C_{i+1}(s)\\
D_{i+1}(s)
\end{pmatrix}
=
\begin{pmatrix}
A(s)&(\lambda-1)^{-1}\\
B(s)&0
\end{pmatrix}
\begin{pmatrix}
C_i(s+1)\\
D_i(s+1)
\end{pmatrix},
\quad \begin{pmatrix}
C_{-1}(s)\\D_{-1}(s)
\end{pmatrix}
:=\begin{pmatrix}0\\1\end{pmatrix}.
\]
Put
\[
E_1^{(r)}(s):=\sum_{i\geq -1} e_i(s+r)C_{r+i}(s),\quad
E_2^{(r)}(s):=\sum_{i\geq -1} e_i(s+r)D_{r+i}(s).
\]
Then
\begin{align*}
&\phi_1(\lambda)\equiv C_1(1-\lambda)^n
[E_1^{(n)}(\mu)H_\mu(\lambda)+
E_2^{(n)}(\mu)H_{\mu-1}(\lambda)],
\\&
\phi_2(\lambda)\equiv C_2(1-\lambda)^{n-1}
[E_1^{(n-1)}(\mu)H_\mu(\lambda)+
E_2^{(n-1)}(\mu)H_{\mu-1}(\lambda)]
\end{align*}
modulo $\ol{\Q(\lambda)}$ with some $C_1,C_2\in\ol\Q^\times$.
Here we note that $E_i^{(r)}(\mu)\in \ol\Q(\lambda)$ are rational functions of variable $\lambda$.
\end{thm}
\begin{proof}
The 3-term relation on ${}_3F_2$ implies that 
$C_i$ and $D_i$ satisfy 
\[
{}_3F_2\left({1,1,1-s-i\atop a,b};x\right)
\equiv
C_i(s,x)
{}_3F_2\left({1,1,1-s\atop a,b};x\right)
+D_i(s,x){}_3F_2\left({1,1,2-s\atop a,b};x\right)
\]
modulo $\Q(s,x)$.
Hence
\begin{align*}
&(1-\lambda)^{m/l+i-1}{}_3F_2
\left(
{1,1,1-m/l-i\atop
2-\alpha^\chi,~2-\beta^\chi}
;(1-\lambda)^{-1}
\right)\\
&\equiv(1-\alpha^\chi)(1-\beta^\chi)(1-\lambda)^{r+i}
(C_{i+r}(\mu,x)H_\mu(\lambda)+D_{i+r}(q^\chi,x)H_{\mu-1}(\lambda))
\end{align*}
for $m=k+lr$, $r\in\Z$.
Apply this to \eqref{reg-sect5-eq5} and \eqref{reg-sect5-eq6}.
The rest is a direct computation (left to the reader).
\end{proof}

\end{document}